\documentclass{amsart}

\usepackage{graphicx}
\usepackage{hyperref}
\usepackage{cite}
\usepackage[utf8]{inputenc}
\usepackage[english]{babel}
\usepackage{amsmath}
\usepackage{amsfonts}
\usepackage{amssymb}
\usepackage{xcolor}
\usepackage{amsthm}
\usepackage[all]{xy}
\usepackage{enumitem}

\usepackage{tikz-cd}
\usetikzlibrary{cd}

\newtheorem{thm}{Theorem}[section]
\newtheorem{lemma}[thm]{Lemma}
\newtheorem{prop}[thm]{Proposition}
\newtheorem{cor}[thm]{Corollary}

\theoremstyle{definition}
\newtheorem{defn}[thm]{Definition}
\newtheorem{example}[thm]{Example}

\theoremstyle{remark}
\newtheorem{rem}[thm]{Remark}

\numberwithin{equation}{section}

\DeclareMathOperator{\Ker}{Ker}

\DeclareMathOperator{\Img}{Im}

\DeclareMathOperator{\End}{End}
\DeclareMathOperator{\Soc}{Soc}
\DeclareMathOperator{\Rad}{Rad}

\newcommand{\ess}{\leq^\text{ess}}

\begin{document}

\title{$\mathfrak{m}$-Endoregular lattices}

\author{Mauricio Medina-B\'arcenas}
\address{Facultad de Ciencias, Universidad Nacional Aut\'onoma de M\'exico, Circuito Exterior,
	Ciudad Universitaria 04510, Ciudad de México, M\'exico}
\email{mmedina@ciencias.unam.mx}
%\curraddr{Department of Mathematics and Statistics,
%Case Western Reserve University, Cleveland, Ohio 43403}

\thanks{The first author was supported by the grant ``Programa de Becas Posdoctorales en la UNAM 2022'' from the Universidad Nacional Aut\'onoma de M\'exico (UNAM)}

\author{Hugo Rinc\'on-Mej\'ia}
\address{Facultad de Ciencias, Universidad Nacional Aut\'onoma de M\'exico, Circuito Exterior,
	Ciudad Universitaria 04510, Ciudad de México, M\'exico}
\email{hurincon@ciencias.unam.mx}

%    General info

%\today

\date{}



\begin{abstract}
	In a previous work, (dual)-$\mathfrak{m}$-Rickart lattices were studied. Now, in this paper, we introduce $\mathfrak{m}$-endoregular lattices as those lattices $\mathcal{L}$ such that $\mathfrak{m}$ is a regular monoid, where $\mathfrak{m}$ is a submonoid with zero of $\End_{lin}(\mathcal{L})$. We show that these lattices can be characterized in terms of $\mathfrak{m}$-Rickart and $\mathfrak{m}$-dual-Rickart lattices. Also, we compare these new lattices with those lattices in which every compact element is a complement. We characterize the $\mathfrak{m}$-endoregular lattices such that every idempotent in $\mathfrak{m}$ is central in $\mathfrak{m}$ and we show that for these lattices the complements are a sublattice which is a Boolean algebra. We introduce two new concepts, $\mathfrak{m}$-$\mathcal{K}$-extending and $\mathfrak{m}$-$\mathcal{T}$-lifting lattices. For these lattices, we show that the monoid $\mathfrak{m}$ has a regular quotient monoid provided they satisfy $\mathfrak{m}$-$C_2$ and $\mathfrak{m}$-$D_2$ respectively.
\end{abstract}

\maketitle

\section{Introduction}
    \emph{Von Neumann regular rings} have been studied widely in the literature for decades. These rings have their roots in functional analysis and they are connected with other classes of rings, for example, \emph{Rickart rings}. In the past years, there has been an interest in studying the notions of von Neumann and Rickart rings in the module theoretic context. Also, once the notion of Rickart module was introduced \cite{leerickart}, its dual naturally emerged \cite{lee2011dual}. In the case of translating the concept of von Neumann regularity to modules, we can find different approaches \cite{leemodules,tuganbaevrings,zelmanowitz1972regular}. In \cite{tuganbaevrings} an $R$-module $M$ is said to be \emph{regular} if every cyclic submodule of $M$ is a direct summand. On the other hand, in \cite{leemodules} an $R$-module $M$ is said to be \emph{endoregular} if the endomorphism ring of $M$ is a von Neumann regular ring. It is clear that these two notions agree when the module is the base ring $R$, but in general, the notions are different. A particularity of endoregular modules is that they can be characterized as modules that are both Rickart and dual Rickart. This allows us to study endoregular modules in terms of Rickart modules and their dual. Subclasses of von Neumann regular rings, such as \emph{unit-regular} and \emph{strongly regular} have also been taken to modules using endoregular modules \cite{medina2020abelian,leeunit}. In all of these cases, many results in rings have been extended to modules and this more general point of view has enriched the theory and opened new lines of study. 
    
	Recall that an $R$-module $M$ is called \emph{Rickart} if $\Ker\varphi$ is a direct summand of $M$ for all $\varphi\in\End_{R}(M)$. Dualy, an $R$-m\'odule $M$ is called \emph{dual-Rickart} if $\Img\varphi$ is a direct summand of $M$ for all $\varphi\in\End_{R}(M)$. In \cite[Proposition 2.3]{leemodules} it is proved that for an $R$-module $M$, $\End_R(M)$ is a von Neumann regular ring if and only if $M$ is Rickart and dual-Rickart. Given an $R$-module $M$, the set of submodules of $M$ is a \emph{complete modular lattice}. In this lattice, the complements are exactly the direct summand of $M$. So, the concepts of Rickart and dual-Rickart modules involve the complements in the lattice of submodules of $M$ and the kernel and image of the homomorphisms. In \cite{albu2013category} it is defined the \emph{linear morphisms} between bounded lattices, mimicking the behavior of a homomorphism of modules in the lattice of submodules. Given a linear morphism $\varphi:\mathcal{L}\to\mathcal{L}'$ there is an element $\ker_{\varphi}\in\mathcal{L}$ called \emph{the kernel of $\varphi$} and there is an image $\varphi(\mathbf{1})\in\mathcal{L}'$ where $\mathbf{1}$ is the greatest element of $\mathcal{L}$ (Definition \ref{linmor}). Therefore it is possible to carry the concepts of Rickart and dual-Rickart to lattices. This has been done in a previous work of the authors \cite{medina2022mbaer}. In that article, it is given a number of properties of the called \emph{Rickart and Baer lattices} and their \emph{duals}. Also, it is studied the relation of the lattice with its monoid of linear endomorphisms. In this way, we defined Rickart and Baer monoids and we proved that these definitions agree with the Rickart and Baer lattices as in the case of Rickart and Baer modules with their endomorphism rings. The approach given in \cite{medina2022mbaer} consisted in to define the concepts and stating the propositions for a submonoid with zero $\mathfrak{m}$ of the monoid of all linear endomorphisms $\End_{lin}(\mathcal{L})$ of a lattice $\mathcal{L}$. With this, we achieve more generality and it is possible to recover much of the theory for modules. 
	
	In this new manuscript, which can be seen as a continuation of \cite{medina2022mbaer}, we consider a complete modular lattice $\mathcal{L}$ and a submonoid with zero $\mathfrak{m}$ of its monoid of linear endomorphism $\End_{lin}(\mathcal{L})$. Our goal is to describe those lattices such that $\mathfrak{m}$ is a \emph{regular monoid}, that is, for every $\varphi\in\mathfrak{m}$ there exists $\psi\in\mathfrak{m}$ such that $\varphi\psi\varphi=\varphi$. A complete modular lattice $\mathcal{L}$ is called \emph{$\mathfrak{m}$-endoregular} if $\mathfrak{m}$ is a regular monoid (Definition \ref{defendoreg}). Let $\Lambda(M)$ denote the lattice of submodules of a module $M$. Given a morphism of modules $f:M\to N$, $f$ induces a linear morphism $f_\ast:\Lambda(M)\to \Lambda(N)$ given by $f_\ast(L)=f(L)$ for all $L\leq M$. Therefore $\mathfrak{E}_M=\{f_\ast\mid f:M\to M\}$ is a submonoid of $\End_{lin}(\Lambda(M))$ which allows us to generalize some of the theory known in modules to lattices, in fact, an $R$-module $M$ is endoregular if and only if $\Lambda(M)$ is $\mathfrak{E}_M$-endoregular. In module theory, there exist the concepts of quasi-continuous and continuous modules (and their duals) which are defined by the ($C_i$)'s (resp. ($D_i$)'s) $1\leq i\leq 3$ conditions \cite{mohamedcontinuous}. A remarkable result about these modules is that if $M$ is a continuous $R$-module and $S$ is its endomorphism ring then $S/\Delta$ is a von Neumann regular ring where $\Delta=\{f\in S\mid \Ker f\ess M\}$ \cite[Proposition 3.5]{mohamedcontinuous}. A similar result is given for a discrete module \cite[Theorem 5.4]{mohamedcontinuous}. In order to take these results to lattices we introduce two new concepts named $\mathfrak{m}$-$\mathcal{K}$-\emph{extending} and $\mathfrak{m}$-$\mathcal{T}$-\emph{lifting lattices} (Definition \ref{defkextlif}). %These new notions generalize the conditions $(C_1)$ and $(D_1)$ given in \cite{albu2016conditions}, respectively. 
    
    The paper is divided into four sections. The first one is this Introduction and in Section \ref{pre} we give some necessary background to make this work self-contained as possible. The concepts and results presented in this section are taken from \cite{albu2013category} and our previous work \cite{medina2022mbaer}. Section \ref{endoreg} is the main part of the paper. In this section, we give general properties of the $\mathfrak{m}$-endoregular lattices. We show that a lattice $\mathcal{L}$ is $\mathfrak{m}$-endoregular if and only if $\mathcal{L}$ is $\mathfrak{m}$-Rickart and dual-$\mathfrak{m}$-Rickart (Theorem \ref{kerimgsumm}). We also introduce the concept of \emph{von Neumann lattice} as a compact lattice $\mathcal{L}$ in which every compact element has a complement as an analogous of regular module in \cite{tuganbaevrings}. We study when an endoregular lattice is von Neumann and vice-versa (Proposition \ref{vnl} and Corollary \ref{vnlendo}). As in the case of modules, we look at the $\mathfrak{m}$-endoregular lattices such that all the idempotents in $\mathfrak{m}$ are central and we call them $\mathfrak{m}$-\emph{abelian-endoregular}. We characterize them (Proposition \ref{ker+img}, Proposition \ref{abendofi}) and we show when in an $\mathfrak{m}$-abelian-endoregular lattice $\mathcal{L}$ the set of complements $C(\mathcal{L})$ is a Boolean sublattice of $\mathcal{L}$ (Proposition \ref{Cboolidemcomm} and Corollary \ref{corCbool}). In the last section, Section \ref{rqm}, we introduce the $\mathfrak{m}$-$\mathcal{K}$-extending and $\mathfrak{m}$-$\mathcal{T}$-lifting lattices (Definition \ref{defkextlif}). We show that every $\mathfrak{m}$-Rickart lattice is $\mathfrak{m}$-$\mathcal{K}$-extending and every $\mathfrak{m}$-dual-Rickart lattice is $\mathfrak{m}$-$\mathcal{T}$-lifting, moreover, we give the converses (Proposition \ref{rickex} and Proposition \ref{drictlif}). We define two congruences $\equiv_\Delta$ and $\equiv^\nabla$ on any submonoid $\mathfrak{m}\subseteq\End_{lin}(\mathcal{L})$. On one hand, it is proved that $\mathfrak{m}/\equiv_\Delta$ is a regular monoid provided that $\mathcal{L}$ is $\mathfrak{m}$-$\mathcal{K}$-extending and satisfies $\mathfrak{m}$-$C_2$, on the other hand, $\mathfrak{m}/\equiv^\nabla$ is a regular monoid provided that $\mathcal{L}$ is $\mathfrak{m}$-$\mathcal{T}$-lifting and satisfies $\mathfrak{m}$-$D_2$ (Theorem \ref{thmDelta} and Theorem \ref{thmnabla}).

	Throughout this paper, $\mathcal{L}$ will denote a (bounded, complete, modular) lattice, the lowest element of $\mathcal{L}$ will be denoted by $\mathbf{0}$ and the greatest element will be denoted by $\mathbf{1}$. Given $a,b\in\mathcal{L}$, $[a,b]$ will denote the interval $\{x\in\mathcal{L}\mid a\leq x\leq b\}$. The subset of complements in $\mathcal{L}$ will be denoted as $C(\mathcal{L})$. The set of linear endomorphisms of $\mathcal{L}$ will be denoted as $\End_{lin}(\mathcal{L})$ which is a monoid with the composition. The letter $\mathfrak{m}$ will stand for a submonoid with zero of $\End_{lin}(\mathcal{L})$. All rings will be associative with unit, and all modules will be left modules. Given an $R$-module, $\End_R(M)$ will denote the endomorphism ring of $M$.
	
\section{Preliminaries}\label{pre}

\begin{defn}\label{linmor}
	A map between bounded lattices $\varphi:\mathcal{L}\to \mathcal{L}'$ is called a \emph{linear morphism} if there exists $\ker_\varphi\in\mathcal{L}$ called the \emph{kernel} of $\varphi$,  and $a'\in\mathcal{L}'$ such that 
	\begin{enumerate}
		\item $\varphi(x)=\varphi(x\vee \ker_\varphi)$ for all $x\in\mathcal{L}$.
		\item $\varphi$ induces an isomorphism of lattices $\overline{\varphi}:[\ker_\varphi,\mathbf{1}]\to [\mathbf{0},a']$ given by $\overline{\varphi}(x)=\varphi(x)$ for all $x\in[\ker_\varphi,\mathbf{1}]$.
	\end{enumerate}
\end{defn}

\begin{rem}\label{linmorjoins}
	If the lattices are complete, we will assume that the isomorphism in item (2) of Definition \ref{linmor} is an isomorphism of complete lattices. In this case, a linear morphism commutes with arbitrary joins \cite[Proposition 1.3]{albu2013category}.
\end{rem}

\textbf{Notation:} Let $\mathcal{L}$ be a complete modular lattice and $a,x\in\mathcal{L}$. There are two canonical linear morphisms $\iota_x:[\mathbf{0},x]\to \mathcal{L}$ the inclusion, and $\rho_a:\mathcal{L}\to[a,\mathbf{1}]$ given by $\rho_a(y)=a\vee y$.

\begin{prop}[{\cite[Proposition 2.4]{medina2022mbaer}}]
	Let $\mathcal{L}$ be a bounded modular lattice, and $x\in\mathcal{L}$ be an element with complement $x'$. Then, the map $\pi_x:\mathcal{L}\to \mathcal{L}$ given by $\pi_x(a)=(a\vee x')\wedge x$ is a linear morphism.
\end{prop}

\begin{defn}
	Let $\mathcal{L}$ be a bounded modular lattice, and $x\in\mathcal{L}$ be an element with a complement. The linear morphism $\pi_x$ is called the \emph{projection on $x$}.
\end{defn}

\begin{rem}\label{exmorf}
	Let $\mathcal{L}$ be a complete modular lattice, and $x,y\in\mathcal{L}$ with $x$ being a complement. Suppose a linear morphism exists $\varphi:[\mathbf{0},x]\to[\mathbf{0},y]$. Then $\varphi$ can be extended to a linear endomorphism of $\mathcal{L}$ considering $\widehat{\varphi}=\iota_y\varphi\pi_x:\mathcal{L}\to \mathcal{L}$.
\end{rem}

\begin{prop}[{\cite[Proposition 2.10]{medina2022mbaer}}]\label{idemcomp}
	Let $\mathcal{L}$ be a bounded modular lattice and $\varepsilon\in\End_{lin}(\mathcal{L})$. If $\varepsilon$ is idempotent, then $\mathbf{0}=\ker_\varepsilon\wedge\varepsilon(\mathbf{1})$ and $\mathbf{1}=\ker_\varepsilon\vee\varepsilon(\mathbf{1})$.
\end{prop}

\begin{cor}\label{idemespi}
	Let $\mathcal{L}$ be a complete modular lattice, and $\varepsilon:\mathcal{L}\to \mathcal{L}$ a linear morphism such that $\varepsilon^2=\varepsilon$. Then $\varepsilon=\pi_{\varepsilon(\mathbf{1})}$.
\end{cor}

\begin{proof}
	Since $\varepsilon(a)=\varepsilon(\varepsilon(a))$ for all $a\in\mathcal{L}$, $a\vee\ker_\varepsilon=\varepsilon(a)\vee\ker_\varepsilon$. Therefore 
	\[\pi_{\varepsilon(\mathbf{1})}(a)=(a\vee\ker_\varepsilon)\wedge\varepsilon(\mathbf{1})=(\varepsilon(a)\vee\ker_\varepsilon)\wedge\varepsilon(\mathbf{1})=\varepsilon(a),\]
	for all $a\in \mathcal{L}$.
\end{proof}

\begin{cor}
		Let $\mathcal{L}$ be a bounded modular lattice. Then there exists a bijective correspondence between idempotent linear  endomorphisms of $\mathcal{L}$ and pairs $(x,x')$ such that $x'$ is a complement of $x$ in $\mathcal{L}$.
\end{cor}

\begin{proof}
	Given an idempotent linear endomorphism $\varepsilon:\mathcal{L}\to \mathcal{L}$, we have the pair $(\ker_\varepsilon,\varepsilon(\mathbf{1}))$. On the other hand, if $(x,x')$ is a pair of elements of $\mathcal{L}$ such that $x'$ is a complement of $x$, then the linear endomorphism $\pi_x(a)=(a\vee x')\wedge x$ is idempotent since $\pi_x$ fixes every element in $[\mathbf{0},x]$.
\end{proof}

Given an $R$-module $M$ and an endomorphism $f:M\to M$, there is a linear morphism $f_\ast:\Lambda(M)\to\Lambda(M)$ induced by $f$. Then, there is a homomorphism of monoids with zero, $(-)_\ast:\End_R(M)\to\End_{lin}(\Lambda(M))$. Let $\mathfrak{E}_M$ denote the image of $\End_R(M)$ under $(-)_\ast$. Then $\mathfrak{E}_M$ is a submonoid with zero of $\End_{lin}(\Lambda(M))$.

\begin{defn}
	Let $\mathcal{L}$ be a complete lattice, and  $\mathfrak{m}$ be a submonoid with zero of $\End_{lin}(\mathcal{L})$.
	\begin{itemize}
		\item $\mathcal{L}$ is called \emph{$\mathfrak{m}$-Rickart} if $\ker_\varphi$ has a complement in $\mathcal{L}$ for all $\varphi\in\mathfrak{m}$.
		\item $\mathcal{L}$ is called \emph{dual-$\mathfrak{m}$-Rickart} if $\varphi(\mathbf{1})$ has a complement in $\mathcal{L}$ for all $\varphi\in\mathfrak{m}$.
	\end{itemize}
	If the submonoid we are considering is $\End_{lin}(\mathcal{L})$, we will omit the $\mathfrak{m}$.
\end{defn}

\section{$\mathfrak{m}$-Endoregular lattices}\label{endoreg}

\begin{defn}
	Let $\mathcal{L}$ be a complete modular lattice and  $\mathfrak{m}$ be a submonoid of $\End_{lin}(\mathcal{L})$. We say that $\mathfrak{m}$ is \emph{closed under complements} if for any $\varphi\in\mathfrak{m}$ and any complements $x,y\in\mathcal{L}$ such that $\varphi$ induces a linear isomorphism $\varphi|_x:[\mathbf{0},x]\cong[\mathbf{0},y]$, it follows that $\iota_x(\varphi|_x)^{-1}\pi_y\in\mathfrak{m}$.
\end{defn}

\begin{defn}
    Given a lattice $\mathcal{L}$ and a submonoid with zero,  $\mathfrak{m}\subseteq\End_{lin}(\mathcal{L})$, we will say that $\mathfrak{m}$ \emph{contains all the projections} if $\pi_x\in\mathfrak{m}$ for every complement $x\in\mathcal{L}$.
\end{defn}

\begin{rem}\label{cucproj}
    If a submonoid $\mathfrak{m}\subseteq\End_{lin}(\mathcal{L})$ is closed under complements, then $\mathfrak{m}$ contains all the projections.
\end{rem}

\begin{prop}
	Let $M$ be an $R$-module. Then $\mathfrak{E}_M$ is closed under complements.
\end{prop}

\begin{proof}
	Let $f\in\End_R(M)$ such that $f_\ast:[0,N]\to[0,L]$ is a linear isomorphism with $N$ and $L$ direct summands of $M$. Since $f_\ast(A)=f(A)=0$ if and only if $A=0$, $f|_N$ is injective. On the other hand $f(N)=L$, therefore $f|_N:N\to L$ is an isomorphism. Let $g:L\to N$ denote the inverse of $f|_N$. Consider the $R$-homomorphism $i(g\oplus 0):M\to M$ where $i:N\hookrightarrow M$ is the canonical inclusion. Then $(i(g\oplus0))_\ast=\iota_Ng_\ast\pi_L$. Thus $\iota_N(f_\ast)^{-1}\pi_L\in\mathfrak{E}_M$.
\end{proof}

\begin{prop}\label{reg}
	Let $\mathcal{L}$ be a complete modular lattice and $\mathfrak{m}$ be a submonoid of $\End_{lin}(\mathcal{L})$ closed under complements. The following conditions are equivalent for $\varphi\in\mathfrak{m}$:
	\begin{enumerate}[label=\emph{(\alph*)}]
		\item $\ker_\varphi$ and $\varphi(\mathbf{1})$ have a complement in $\mathcal{L}$.
		\item There exists $\psi\in\mathfrak{m}$ such that $\varphi=\varphi\psi\varphi$.
	\end{enumerate}
	Moreover, $\psi\varphi(\mathbf{1})$ is a complement of $\ker_\varphi$, and $\ker_{\psi\varphi}$ is a complement of $\varphi(\mathbf{1})$.
\end{prop}

\begin{proof}
	(a)$\Rightarrow$(b) By hypothesis, there exists $x\in\mathcal{L}$ such that $\mathbf{1}=\ker_\varphi\vee x$ and $\mathbf{0}=\ker_\varphi\wedge x$. Then, $\varphi$ induces a linear isomorphism $\varphi|_x:[\mathbf{0},x]\to[\mathbf{0},\varphi(\mathbf{1})]$. Note that $\varphi(x)=\varphi(\mathbf{1})$. Define $\psi:\mathcal{L}\to\mathcal{L}$ as $\psi=\iota_x(\varphi|_x)^{-1}\pi_{\varphi(\mathbf{1})}$. By hypothesis, $\psi\in\mathfrak{m}$. Hence	\[\varphi\psi\varphi(a)=\varphi\iota_x(\varphi|_x)^{-1}\pi_{\varphi(\mathbf{1})}\varphi(a)=\varphi\iota_x(\varphi|_x)^{-1}\varphi(a)=\varphi(a).\]
	
	(b)$\Rightarrow$(a) Suppose that there exists $\psi\in\mathfrak{m}$ such that $\varphi=\varphi\psi\varphi$. Note that $\varphi\psi$ and $\psi\varphi$ are idempotent elements of $\mathfrak{m}$. We have that 
	\[\varphi\psi(\mathbf{1})\leq\varphi(\mathbf{1})=\varphi\psi(\varphi(\mathbf{1}))\leq\varphi\psi(\mathbf{1}).\]
	This implies that $\varphi\psi(\mathbf{1})=\varphi(\mathbf{1})$. It follows from Proposition \ref{idemcomp} that $\mathbf{1}=\varphi\psi(\mathbf{1})\vee\ker_{\varphi\psi}=\varphi(\mathbf{1})\vee\ker_{\varphi\psi}$ and $\mathbf{0}=\varphi\psi(\mathbf{1})\wedge\ker_{\varphi\psi}=\varphi(\mathbf{1})\wedge\ker_{\varphi\psi}$. Therefore $\varphi(\mathbf{1})$ is a complement in $\mathcal{L}$. Now, let $\overline{\varphi}:[\ker_\varphi,\mathbf{1}]\to[\mathbf{0},\varphi(\mathbf{1})]$ be the isomorphism induced by $\varphi$. Then,
	\[\overline{\varphi}(\psi\varphi(\mathbf{1})\vee\ker_\varphi)=\varphi(\psi\varphi(\mathbf{1})\vee\ker_\varphi)=\varphi\psi\varphi(\mathbf{1})=\varphi(\mathbf{1})=\overline{\varphi}(\mathbf{1}).\]
	Thus, $\psi\varphi(\mathbf{1})\vee\ker_\varphi=\mathbf{1}$. Note that $\ker_\varphi\leq\ker_{\psi\varphi}$ and $\ker_{\psi\varphi}\wedge\psi\varphi(\mathbf{1})=\mathbf{0}$. Hence, $\ker_\varphi\wedge\psi\varphi(\mathbf{1})=\mathbf{0}$.
\end{proof}

\begin{cor}\label{regM}
    Let $M$ be an $R$-module. The following conditions are equivalent for $f\in\End_R(M)$:
    	\begin{enumerate}[label=\emph{(\alph*)}]
		\item $\Ker f$ and $f(M)$ are direct summands of $M$.
		\item There exists $g\in\End_R(M)$ such that $f(N)=fgf(N)$ for all $N\leq M$.
	\end{enumerate}
	Moreover, $gf(M)$ is a complement of $\Ker f$, and $\Ker fg$ is a complement of $f(M)$.
\end{cor}

\begin{defn}
A monoid $E$ is said to be \emph{regular} if, for any $\varphi\in E$, there exists $\psi\in E$ such that $\varphi=\varphi\psi\varphi$.
\end{defn}

\begin{defn}\label{defendoreg}
	Let $\mathcal{L}$ be a complete lattice and $\mathfrak{m}$ be a submonoid with zero of $\End_{lin}(\mathcal{L})$. The lattice $\mathcal{L}$ is called \emph{$\mathfrak{m}$-endoregular} if the monoid $\mathfrak{m}$ is regular. If the submonoid we are considering is $\End_{lin}(\mathcal{L})$, we will omit the $\mathfrak{m}$.
\end{defn}

An $R$-module $M$ is said to be \emph{endoregular} if $\End_{R}(M)$ is a von Neumann regular ring \cite{leemodules}. This definition can be compared with Definition \ref{defendoreg} as follows:

\begin{cor}
The following conditions are equivalent for an $R$-module $M$:
\begin{enumerate}[label=\emph{(\alph*)}]
    \item $M$ is endoregular.
    \item $\Lambda(M)$ is a $\mathfrak{E}_M$-endoregular lattice.
    \item For all $f\in\End_R(M)$, there exists $g\in\End_R(M)$ such that $f(N)=fgf(N)$ for all $N\leq M$.
    \item $\Ker f$ and $f(M)$ are direct summands of $M$ for all $f\in\End_R(M)$.
\end{enumerate}
\end{cor}

\begin{proof}
	(a)$\Rightarrow$(b) Let $f_\ast\in\mathfrak{E}_M$. By hypothesis, there exists $g\in\End_R(M)$ such that $f=fgf$. Then $f_\ast=(fgf)_\ast=f_\ast g_\ast f_\ast$. Thus $\mathfrak{E}_M$ is a regular monoid.
	
	(b)$\Rightarrow$(c) It is clear.
	
	(c)$\Rightarrow$(d) It follows from Corollary \ref{regM}. 
	
	(d)$\Rightarrow$(a) It follows from \cite[Proposition 2.3]{leemodules}
\end{proof}

We mention the following definitions taken from \cite{medina2022mbaer}.

\begin{defn}\label{defc2}
	Let $\mathcal{L}$ be a complete modular lattice, $a,x,x'\in\mathcal{L}$ with $x'$ a complement of $x$. It is said that $\mathcal{L}$ satisfies \emph{$\mathfrak{m}$-$C_2$ condition} if whenever there is an isomorphism $\theta:[\mathbf{0},x]\overset{\cong}{\rightarrow}[\mathbf{0},a]$, and $\iota_a\theta\pi_{x}$ is in $\mathfrak{m}$, this implies that $a$ is a complement.
\end{defn}

\begin{defn}
	Let $\mathcal{L}$ be a complete modular lattice, $a,x\in\mathcal{L}$, and $\mathfrak{m}\subseteq\End_{lin}(\mathcal{L})$ be a submonoid. It is said that $\mathcal{L}$ satisfies \emph{$\mathfrak{m}$-$D_2$ condition} if whenever there is an isomorphism $\theta:[a,\mathbf{1}]\overset{\cong}{\rightarrow}[\mathbf{0},x]$, with $x$ a complement in $\mathcal{L}$ and $\iota_x\theta\rho_a \in \mathfrak{m}$,  then $a$ is a complement.
\end{defn}

\begin{thm}\label{kerimgsumm}
	Let $\mathcal{L}$ be a complete modular lattice and $\mathfrak{m}$ be a submonoid of $\End_{lin}(\mathcal{L})$ closed under complements. The following conditions are equivalent:
	\begin{enumerate}[label=\emph{(\alph*)}]
		\item $\mathcal{L}$ is $\mathfrak{m}$-endoregular.
		\item $\mathcal{L}$ is $\mathfrak{m}$-Rickart and satisfies the condition $\mathfrak{m}$-$C_2$.
		\item $\mathcal{L}$ is dual-$\mathfrak{m}$-Rickart and satisfies the condition $\mathfrak{m}$-$D_2$.
		\item $\ker_\varphi$ and $\varphi(\mathbf{1})$ have a complement in $\mathcal{L}$ for every $\varphi\in\mathfrak{m}$.
	\end{enumerate}
\end{thm}

\begin{proof}
	(a)$\Leftrightarrow$(d) follows from Proposition \ref{reg}.
	
	(b)$\Rightarrow$(d) Let $\varphi\in\mathfrak{m}$. Since $\mathcal{L}$ is $\mathfrak{m}$-Rickart, $\ker_\varphi$ is a complement. Let $x\in\mathcal{L}$ be a complement of $\ker_\varphi$. Then, there is an isomorphism $\theta=\overline{\varphi}(\ker_{\varphi}\vee\_):[\mathbf{0},x]\to[\ker_\varphi,\mathbf{1}]\to[\mathbf{0},\varphi(1)]$. We claim that $\iota_{\varphi(1)}\theta\pi_x\in\mathfrak{m}$. Let $y\in\mathcal{L}$. It follows that
	\[\iota_{\varphi(1)}\theta\pi_x(y)=\iota_{\varphi(1)}\theta((y\vee\ker_\varphi)\wedge x)=\iota_{\varphi(1)}\overline{\varphi}(\ker_{\varphi}\vee\_)((y\vee\ker_\varphi)\wedge x)\]
	\[=\iota_{\varphi(1)}\overline{\varphi}(y\vee\ker_\varphi)=\varphi(y).\]
	Therefore $\iota_{\varphi(1)}\theta\pi_x=\varphi\in\mathfrak{m}$. This implies, by the $\mathfrak{m}$-$C_2$ condition, that $\varphi(\mathbf{1})$ is a complement in $\mathcal{L}$.
	
	(d)$\Rightarrow$(b) By hypothesis $\mathcal{L}$ is $\mathfrak{m}$-Rickart. Let $a,x,x'\in\mathcal{L}$ with $x'$ a complement of $x$. Suppose that there exists an isomorphism  $\theta:[\mathbf{0},x]\to[\mathbf{0},a]$ with $\iota_a\theta\pi_x\in\mathfrak{m}$. Let $\varphi$ denote the composition $\iota_a\theta\pi_x$, then $\varphi(\mathbf{1})=a$. By hypothesis $a$ is a complement. Thus $\mathcal{L}$ satisfies the $\mathfrak{m}$-$C_2$ condition.
	(c)$\Leftrightarrow(d)$ can be proved similarly. 
\end{proof}

\begin{cor}
    Let $\mathcal{L}$ be a complete modular lattice and $\mathfrak{m}$ be a submonoid of $\End_{lin}(\mathcal{L})$ closed under complements. The following conditions are equivalent:
    \begin{enumerate}[label=\emph{(\alph*)}]
		\item $\mathcal{L}$ is $\mathfrak{m}$-endoregular.
		\item $\mathcal{L}$ is $\mathfrak{m}$-Rickart and dual-$\mathfrak{m}$-Rickart.
    \end{enumerate}
\end{cor}

\begin{proof}
    (a)$\Rightarrow$(b) follows from Theorem \ref{kerimgsumm}. For (b)$\Rightarrow$(a), we have that an $\mathfrak{m}$-Rickart lattice satisfies $\mathfrak{m}$-$D_2$ condition \cite[Proposition 3.42]{medina2022mbaer} and a dual-$\mathfrak{m}$-Rickart lattice satisfies $\mathfrak{m}$-$C_2$ condition \cite[Proposition 3.43]{medina2022mbaer}.
\end{proof}

\begin{cor}
	Let $\mathcal{L}$ be an indecomposable modular lattice, that is, $C(\mathcal{L})=\{\mathbf{0},\mathbf{1}\}$ and $\mathfrak{m}\subseteq\End_{lin}(\mathcal{L})$ a submonoid. The following conditions are equivalent:
	\begin{enumerate}[label=\emph{(\alph*)}]
		\item $\mathcal{L}$ is $\mathfrak{m}$-endoregular.
		\item Every $0\neq \varphi\in\mathfrak{m}$ has an inverse.
	\end{enumerate}
\end{cor}

\begin{defn}
	A complete lattice $\mathcal{L}$ is \emph{compact} if whenever $\mathbf{1}=\bigvee_{i\in I}a_i$, there exists a finite subset $F\subseteq I$ such that $\mathbf{1}=\bigvee_{i\in F}a_i$. An element $c$ in a complete lattice is \emph{compact} if $[\mathbf{0},c]$ is a compact lattice.
\end{defn}

\begin{rem}
	Given an $R$-module $M$, the lattice $\Lambda(M)$ is compact if and only if $M$ is finitely generated. Similarly, a submodule is compact if and only if it is finitely generated. 
\end{rem}

\begin{defn}
	A complete lattice $\mathcal{L}$ is \emph{von Neumann regular} if every compact element of $\mathcal{L}$ has a complement.
\end{defn}

\begin{rem}
	For an $R$-module $M$, the lattice $\Lambda(M)$ is von Neumann regular if and only if every finitely generated (cyclic) submodule of $M$ is a direct summand. In \cite{tuganbaevrings}, A. Tuganbaev calls these modules \emph{regular}. When $M=R$, this definition agrees with that of a \emph{von Neumann regular ring}.
\end{rem}

\begin{lemma}\label{imgcomp}
	Let $\mathcal{L}$ be a complete compact modular lattice. Then $\varphi(\mathbf{1})$ is compact for every linear morphism $\varphi:\mathcal{L}\to\mathcal{G}$.
\end{lemma}

\begin{proof}
	Suppose $\varphi(\mathbf{1})=\bigvee_{i\in I}a_i$. Consider the isomorphism $\overline{\varphi}:[\ker_\varphi,\mathbf{1}]\to [\mathbf{0},\varphi(\mathbf{1})]$. Then 
	\[\mathbf{1}={\overline{\varphi}}^{-1}\varphi(\mathbf{1})={\overline{\varphi}}^{-1}\left(\bigvee_{i\in I}a_i \right)=\bigvee_{i\in I}{\overline{\varphi}}^{-1}(a_i).\]
	By hypothesis, a finite subset $F\subseteq I$ exists, such that $\mathbf{1}=\bigvee_{i\in F}{\overline{\varphi}}^{-1}(a_i)$. Applying $\varphi$ we get that $\varphi(\mathbf{1})=\bigvee_{i\in F}a_i$. Thus, $\varphi(\mathbf{1})$ is compact.
\end{proof}

Given a complete modular lattice $\mathcal{L}$ and a submonoid $\mathfrak{m}\subseteq\End_{lin}(\mathcal{L})$, it is said that an element $a\in \mathcal{L}$ is $\mathfrak{m}$-$\mathcal{L}$-\emph{generated} if there exists a family of linear morphisms $\{\varphi_i:\mathcal{L}\to[\mathbf{0},a]\}_I$ such that $\iota_a\varphi_i\in\mathfrak{m}$ for all $i\in I$ and $\bigvee_I\varphi_i(\mathbf{1})=a$ \cite[Definition 3.48]{medina2022mbaer}.

\begin{prop}\label{vnl}
	Let $\mathcal{L}$ be a complete compact modular lattice and $\mathfrak{m}$ be a submonoid of $\End_{lin}(\mathcal{L})$ closed under complements. The following conditions are equivalent:
	\begin{enumerate}[label=\emph{(\alph*)}]
		\item $\mathcal{L}$ is von Neumann regular.
		\item $\mathcal{L}$ is dual-$\mathfrak{m}$-Rickart, and every compact element is $\mathfrak{m}$-$\mathcal{L}$-generated.
	\end{enumerate}
\end{prop}

\begin{proof}
	(a)$\Rightarrow$(b) Let $\varphi\in\mathfrak{m}$. It follows from Lemma \ref{imgcomp} that $\varphi(\mathbf{1})$ is compact. By hypothesis $\varphi(\mathbf{1})$ has a complement; that is, $\mathcal{L}$ is dual-$\mathfrak{m}$-Rickart. Now, let $c\in\mathcal{L}$ be a compact element. By hypothesis, $c$ has a complement, then $\pi_c(\mathbf{1})=c$, and $\mathfrak{m}$ contains all the projections. Thus, $c$ is $\mathfrak{m}$-$\mathcal{L}$-generated.
	
	(b)$\Rightarrow$(a) Let $c\in\mathcal{L}$ be a compact element. By hypothesis, there exists a family of linear morphisms $\{\varphi_i:\mathcal{L}\to[\mathbf{0},c]\}_I$ such that $\iota_c\varphi_i\in\mathfrak{m}$ for all $i\in I$ and $\bigvee_{i\in I}\varphi_i(\mathbf{1})=c$. Since $c$ is compact, a finite subset $F\subseteq I$ exists, such that $c=\bigvee_{i\in F}\varphi_i(\mathbf{1})$. It follows from \cite[Proposition 3.24]{medina2022mbaer} that $c$ is a complement because each $\varphi_i(\mathbf{1})$ is a complement in $\mathcal{L}$.
\end{proof}

\begin{cor}
	The following conditions are equivalent for a finitely generated module $M$:
	\begin{enumerate}[label=\emph{(\alph*)}]
		\item $M$ is regular (in the sense of \cite{tuganbaevrings}).
		\item $M$ is dual-Rickart and generates all its cyclic submodules.
		\item $M$ is dual-Rickart and generates all its finitely generated submodules.
	\end{enumerate}
\end{cor}

Using Theorem \ref{kerimgsumm}, we get the following corollaries.

\begin{cor}\label{vnlendo}
	Let $\mathcal{L}$ be a complete compact modular lattice and $\mathfrak{m}$ be a submonoid of $\End_{lin}(\mathcal{L})$ closed under complements. The following conditions are equivalent:
	\begin{enumerate}[label=\emph{(\alph*)}]
		\item $\mathcal{L}$ is von Neumann regular and satisfies the condition $\mathfrak{m}$-$D_2$.
		\item $\mathcal{L}$ is $\mathfrak{m}$-Rickart, dual-$\mathfrak{m}$-Rickart, and every compact element is $\mathfrak{m}$-$\mathcal{L}$-generated.
		\item $\mathcal{L}$ is $\mathfrak{m}$-endoregular, and every compact element is $\mathfrak{m}$-$\mathcal{L}$-generated.
	\end{enumerate}
\end{cor}

\begin{cor}
	The following conditions are equivalent for a finitely generated module $M$:
	\begin{enumerate}[label=\emph{(\alph*)}]
		\item $M$ is regular (in the sense of \cite{tuganbaevrings}) and satisfies the $(D_2)$ condition.
		\item $M$ is Rickart, dual-Rickart, and generates all its cyclic submodules.
		\item $M$ is endoregular and generates all its cyclic submodules.
	\end{enumerate}
\end{cor}

Recall that a left ideal $A$ in a monoid $E$ is a subset such that $EA\subseteq A$. If the monoid $E$ is a monoid with zero, we ask that $0\in A$. Right, and two-sided ideals in a monoid are defined similarly.

\begin{lemma}\label{regsp}
	Let $E$ be a regular monoid with zero, and let $A$ be a left ideal of $\mathfrak{m}$. If $A^2=0$,  then $A=0$. 
\end{lemma}

\begin{proof}
	Let $x\in A$. By hypothesis, there exists $y\in E$ such that $xyx=x$. Hence $x\in A^2=0$.
\end{proof}

\begin{lemma}\label{pife0central}
	Let $\mathcal{L}$ be an $\mathfrak{m}$-endoregular lattice and $\varepsilon^2=\varepsilon\in\mathfrak{m}$ be idempotent.  Then, $\varepsilon$ is central in $\mathfrak{m}$ if and only if $\pi_{\ker_\varepsilon}\varphi\varepsilon=0$ for all $\varphi\in\mathfrak{m}$. 
\end{lemma}

\begin{proof}
	Set $k=\ker_\varepsilon$ and suppose $\pi_{k}\varphi\varepsilon=0$ for all $\varphi\in\mathfrak{m}$. We claim that $\varepsilon\varphi\pi_k=0$ for all $\varphi\in\mathfrak{m}$. Let $\varphi\in\mathfrak{m}$. By hypothesis, $\varphi\varepsilon(a)\leq\varepsilon(\mathbf{1})$ for all $a\in\mathcal{L}$. By Corollary \ref{idemespi}, $\varepsilon(\varphi\varepsilon(a))=\varphi\varepsilon(a)$ for all $a\in \mathcal{L}$. Hence $\varepsilon\varphi\varepsilon=\varphi\varepsilon$. It follows that $\mathfrak{m}\varepsilon\subseteq\varepsilon\mathfrak{m}$. Thus, $\mathfrak{m}(\varepsilon\mathfrak{m}\pi_k)\subseteq\varepsilon\mathfrak{m}\pi_k$. Also, $(\varepsilon\mathfrak{m}\pi_k)^2=0$. Since $\mathfrak{m}$ is a regular monoid, then $\varepsilon\mathfrak{m}\pi_k=0$ by Lemma \ref{regsp}, proving the claim.
	
	Let $a\in\mathcal{L}$ and $\varphi\in\mathfrak{m}$. By \cite[Proposition 2.14]{medina2022mbaer} and Corollary \ref{idemespi}, $a\leq\pi_k(a)\vee\varepsilon(a)$. Then $\varphi(a)\leq\varphi\pi_k(a)\vee\varphi\varepsilon(a)$. Thus $\varepsilon\varphi(a)\leq\varepsilon\varphi\pi_k(a)\vee\varepsilon\varphi\varepsilon(a)=\varepsilon\varphi\varepsilon(a)$. % On the other hand, note that $\mathbf{1}=k\vee\varepsilon(\mathbf{1})$, implies that
 On the other hand,   $(a\wedge k)\vee(a\wedge\varepsilon(\mathbf{1}))\leq a$ for all $a\in\mathcal{L}$. Thus,
	\[\varepsilon\varphi\varepsilon(a)=\varepsilon\varphi((a\vee k)\wedge k)\vee\varepsilon\varphi((a\vee k)\wedge\varepsilon(\mathbf{1}))\leq \varepsilon\varphi(a\vee k)=\varepsilon\varphi(a),\]
	for all $a\in\mathcal{L}$. Hence $\varepsilon\varphi=\varepsilon\varphi\varepsilon=\varphi\varepsilon$.
\end{proof}

\begin{defn}
    Let $\mathcal{L}$ be a bounded lattice and $\mathfrak{m}\subseteq\End_{lin}(\mathcal{L})$ be a submonoid. $\mathcal{L}$ is called $\mathfrak{m}$-\emph{abelian} if the idempotents in $\mathfrak{m}$ are central in $\mathfrak{m}$. If the monoid we are considering is $\End_{lin}(\mathcal{L})$ we will omit the $\mathfrak{m}$.
\end{defn}

\begin{prop}\label{ker+img}
	Let $\mathcal{L}$ be a complete modular lattice and $\mathfrak{m}$ be a submonoid of $\End_{lin}(\mathcal{L})$ closed under complements. The following conditions are equivalent:
	\begin{enumerate}[label=\emph{(\alph*)}]
		\item $\mathcal{L}$ is $\mathfrak{m}$-Rickart, dual-$\mathfrak{m}$-Rickart, and $\mathfrak{m}$-abelian.
		\item $\mathbf{0}=\varphi(\mathbf{1})\wedge\ker_\varphi$ and $\mathbf{1}=\ker_\varphi\vee\varphi(\mathbf{1})$ for all $\varphi\in\mathfrak{m}$.
	\end{enumerate}
\end{prop}

\begin{proof}
	(a)$\Rightarrow$(b) Let $\varphi\in\mathfrak{m}$. Let $k'\in\mathcal{L}$ be a complement of $k=\ker_\varphi$. Then $\mathbf{0}=\varphi(k)=\varphi\pi_k(\mathbf{1})=\pi_k\varphi(\mathbf{1})$. This implies that $\varphi(\mathbf{1})\leq k'$. Therefore $\varphi(\mathbf{1})\wedge k\leq k'\wedge k=\mathbf{0}$. Now, from Theorem \ref{kerimgsumm} there exists $\psi\in\mathfrak{m}$ such that $\varphi\psi\varphi=\varphi$. Since $\varphi(\mathbf{1})=(\varphi\psi)\varphi(\mathbf{1})=\varphi\varphi\psi(\mathbf{1})$, $\mathbf{1}=\varphi\psi(\mathbf{1})\vee k$. On the other hand, let $x\in\mathcal{L}$ be a complement of $\varphi(\mathbf{1})$. Then,
	\[\varphi\psi(\mathbf{1})=\varphi\psi(\varphi(\mathbf{1})\vee x)=\varphi\psi\varphi(\mathbf{1})\vee\varphi\psi(x)=\varphi(\mathbf{1})\vee\varphi\psi(x)=\varphi(\mathbf{1}\vee\psi(x))=\varphi(\mathbf{1}).\]
	Thus, $\mathbf{1}=\varphi(\mathbf{1})\vee k$.
	
	(b)$\Rightarrow$(a) By Theorem \ref{kerimgsumm}, we have that $\mathcal{L}$ is $\mathfrak{m}$-Rickart and dual-$\mathfrak{m}$-Rickart. Let $\varepsilon,\varphi\in\mathfrak{m}$ with $\varepsilon^2=\varepsilon$. Let $k=\ker_\varepsilon$
	%. Then, $\mathbf{0}=k\wedge \varepsilon(\mathbf{1})$ and $\mathbf{1}=k\vee\varepsilon(\mathbf{1})$. Since $\varepsilon$ is idempotent, $a\vee k=\varepsilon(a)\vee k$ for all $a\in\mathcal{L}$. Therefore, 
	%\[\varepsilon\varphi(a)=\varepsilon\varphi(a\vee k)=\varepsilon\varphi(\varepsilon(a)\vee k)=\varepsilon\varphi\varepsilon(a)\]
	%for all $a\in \mathcal{L}$. Thus, $\varepsilon\varphi=\varepsilon\varphi\varepsilon$. On the other hand, 
	and consider $\alpha=\pi_k\varphi\varepsilon\in\mathfrak{m}$. Note that $\alpha^2=0$ because $\varepsilon\pi_k=0$ and by hypothesis $\mathbf{1}=\alpha(\mathbf{1})\vee\ker_\alpha$. Hence $\alpha(\mathbf{1})=\mathbf{0}$, that is, $\mathbf{0}=\alpha(\mathbf{1})=\pi_k(\varphi\varepsilon(\mathbf{1}))$. Thus $\varphi\varepsilon=\varepsilon\varphi$ by Lemma \ref{pife0central}.
\end{proof}

\begin{cor}\label{compabend}
	Let $\mathcal{L}$ be a complete modular lattice, $a\in\mathcal{L}$ be a complement, $\mathfrak{m}$ be a submonoid of $\End_{lin}(\mathcal{L})$, and let $\mathfrak{n}$ be a submonoid of $\End_{lin}([\mathbf{0},a])$ such that $\iota_a\psi\pi_a\in\mathfrak{m}$ for every element $\psi\in\mathfrak{n}$. Suppose that the monoids $\mathfrak{m}$ and $\mathfrak{n}$ are closed under complements. If $\mathcal{L}$ is  $\mathfrak{m}$-abelian-endoregular, then $[\mathbf{0},a]$ is an  $\mathfrak{n}$-abelian-endoregular lattice.
\end{cor}

\begin{proof}
	Let $\psi\in\mathfrak{n}$. We have that $\varphi=\iota_a\psi\pi_a\in\mathfrak{m}$. Then $\mathbf{0}=\varphi(\mathbf{1})\wedge\ker_\varphi$ and $\mathbf{1}=\ker_\varphi\vee\varphi(\mathbf{1})$ by Proposition \ref{ker+img}. Note that $\ker_{\varphi}=\ker_{\psi}\vee x$ where $x$ is a complement of $a$ in $\mathcal{L}$ and $\varphi(\mathbf{1})=\psi(a)$. Hence,
	\[a=a\wedge(\ker_\varphi\vee\varphi(\mathbf{1}))=a\wedge(\ker_{\psi}\vee x\vee \psi(a))=(\ker_{\psi}\vee\psi(a))\vee(a\wedge x)=\ker_{\psi}\vee\psi(a).\]
	On the other hand, $\ker_{\psi}\wedge\psi(a)\leq \ker_{\varphi}\wedge\varphi(\mathbf{1})=\mathbf{0}$. Thus $[\mathbf{0},a]$ is an  $\mathfrak{n}$-abelian-endoregular lattice by Proposition \ref{ker+img}.
\end{proof}

\begin{prop}\label{abendofi}
	Let $\mathcal{L}$ be a complete modular lattice and $\mathfrak{m}$ be a submonoid of $\End_{lin}(\mathcal{L})$ closed under complements. The following conditions are equivalent:
	\begin{enumerate}[label=\emph{(\alph*)}]
		\item $\mathcal{L}$ is $\mathfrak{m}$-abelian-endoregular.
		\item $\mathcal{L}$ is $\mathfrak{m}$-endoregular, and $\varphi(a)\leq a$ for every $\mathfrak{m}$-$\mathcal{L}$-generated element $a\in \mathcal{L}$ and for all $\varphi\in \mathfrak{m}$.
	\end{enumerate}
\end{prop}

\begin{proof}
	(a)$\Rightarrow$(b) Let $a\in\mathcal{L}$ be $\mathfrak{m}$-$\mathcal{L}$-generated, that is, there exists a family of linear morphisms $\{\varphi_i:\mathcal{L}\to[\mathbf{0},a]\}_I$ such that $\iota_a\varphi_i\in\mathfrak{m}$ for all $i\in I$ and $\bigvee_{i\in I}\varphi_i(\mathbf{1})=a$. Let $\psi\in\mathfrak{m}$. Then $\psi(a)=\psi\left(\bigvee_{i\in I}\varphi_i(\mathbf{1})\right)=\bigvee_{i\in I}\psi\varphi_i(\mathbf{1})$. Fix $i\in I$. Then, there exists $\varphi'\in\mathfrak{m}$ such that $\varphi_i\varphi'\varphi_i=\varphi_i$. By hypothesis, 
	\[\psi\varphi_i(\mathbf{1})=\psi\varphi_i\varphi'\varphi_i(\mathbf{1})=\varphi_i\varphi'\psi\varphi_i(\mathbf{1})\leq\varphi_i(\mathbf{1})\leq a.\]
	Thus, $\psi\varphi_i(\mathbf{1})\leq a$ for all $i\in I$. This implies that $\psi(a)=\bigvee_{i\in I}\psi\varphi_i(\mathbf{1})\leq a$.
	
	(b)$\Rightarrow$(a) Let $\varepsilon^2=\varepsilon\in\mathfrak{m}$. Then $\varepsilon(\mathbf{1})$ is $\mathfrak{m}$-$\mathcal{L}$-generated. Let $\varphi\in\mathfrak{m}$, since $\varphi\varepsilon(1)\leq\varepsilon(\mathbf{1})$, we have that $\varepsilon\varphi\varepsilon(a)=\varphi\varepsilon(a)$ for all $a\in\mathcal{L}$. Hence $\pi_{\ker_\varepsilon}\varphi\varepsilon=0$. By Lemma \ref{pife0central}, $\varepsilon$ is central in $\mathfrak{m}$.
\end{proof}

Given a von Neumann regular ring $R$, it is known that $R$ is abelian (i.e. the idempotents are central) if and only if the lattice of left direct summands of $R$ is Boolean \cite[Theorem 3.4]{goodearlneumann} and \cite[Proposition 3.3]{cualuguareanu2010modules} (Also this result is still true for modules, see  \cite[Proposition 2.11]{medina2020abelian}). In the case of lattices, this result is no longer true in general, as the following example shows.

\begin{example}\label{idemnocentral}
	Let $\mathcal{L}$ be the lattice given by the following diagram:
	
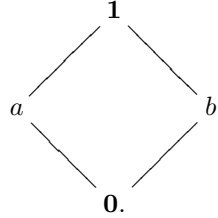
\[\xymatrix{ & \mathbf{1}\ar@{-}[dl]\ar@{-}[dr] & \\ a\ar@{-}[dr] & & b\ar@{-}[dl] \\ & \mathbf{0}. & }\]
	Then $\mathcal{L}$ is a Boolean algebra, in particular,  $C(\mathcal{L})=\mathcal{L}$.
	Consider the  following linear endomorphisms of $\mathcal{L}$:
	\[\begin{array}{c c c c c}
	\tau(\mathbf{0})=\mathbf{0} & &  \varphi(\mathbf{0})=\mathbf{0} & &  \psi(\mathbf{0})=\mathbf{0} \\
	\tau(a)=b & & \varphi(a)=a & & \psi(a)=\mathbf{0} \\
	\tau(b)=a & & \varphi(b)=\mathbf{0} & & \psi(b) =b \\
	\tau(\mathbf{1})=\mathbf{1} & & \varphi(\mathbf{1})=a & & \psi(\mathbf{1})=b.
	\end{array}\]
	Then $\End_{lin}(\mathcal{L})=\{Id,0,\tau,\varphi,\psi,\tau\varphi,\varphi\tau\}$. Note that $\varphi^2=\varphi$, $\psi^2=\psi$ and $\psi\varphi=\varphi\psi$. But $\tau\varphi\neq\varphi\tau$. Thus, $\mathcal{L}$ is an endoregular lattice which is not abelian. If we take $\mathfrak{m}=\{Id,0,\varphi,\psi\}$, then $\mathcal{L}$ is $\mathfrak{m}$-abelian-endoregular.
\end{example}

\begin{prop}\label{Cboolidemcomm}
	Let $\mathcal{L}$ be a complete modular lattice and $\mathfrak{m}$ be a submonoid of $\End_{lin}(\mathcal{L})$ closed under complements. Suppose $\mathcal{L}$ is $\mathfrak{m}$-endoregular. The following conditions are equivalent:
	\begin{enumerate}[label=\emph{(\alph*)}]
		\item Any two idempotents of $\mathfrak{m}$ commute.
		\item $C(\mathcal{L})$ is Boolean.
	\end{enumerate}
\end{prop}

\begin{proof}
	 Note that $C(\mathcal{L})$ is a sublattice of $\mathcal{L}$ because $\mathcal{L}$ is $\mathfrak{m}$-Rickart and dual-$\mathfrak{m}$-Rickart \cite[Proposition 3.24]{medina2022mbaer}.
	 
	 (a)$\Rightarrow$(b) Let $x,y\in C(\mathcal{L})$. Then $\pi_x(y)=\pi_x\pi_y(\mathbf{1})=\pi_y\pi_x(\mathbf{1})=\pi_y(x)\leq x\wedge y$. We always have that $x\wedge y\leq\pi_x(y)$. Hence $\pi_x(y)=x\wedge y=\pi_y(x)$. It follows from \cite[Proposition 3.29]{medina2022mbaer} that $C(\mathcal{L})$ is a Boolean algebra.
	
	(b)$\Rightarrow$(a) Let $\alpha,\varepsilon\in\mathfrak{m}$ be two idempotents. We have that $\mathbf{1}=\alpha(\mathbf{1})\vee\ker_\alpha$. By hypothesis, $\varepsilon(\mathbf{1})=\varepsilon(\mathbf{1})\wedge(\alpha(\mathbf{1})\vee\ker_\alpha)=(\varepsilon(\mathbf{1})\wedge\alpha(\mathbf{1}))\vee(\varepsilon(\mathbf{1})\wedge\ker_\alpha)$. Therefore, $\alpha\varepsilon(\mathbf{1})=\alpha(\varepsilon(\mathbf{1})\wedge\alpha(\mathbf{1}))=\varepsilon(\mathbf{1})\wedge\alpha(\mathbf{1})\leq\varepsilon(\mathbf{1})$. This implies that $\varepsilon\alpha\varepsilon=\alpha\varepsilon$. Thus $\pi_{\ker_\varepsilon}\alpha\varepsilon=0$. Analogously, $\varepsilon\alpha\pi_{\ker_\varepsilon}=0$. Hence $\alpha(\ker_\varepsilon)\leq\ker_\varepsilon$. It follows that
	\[\varepsilon\alpha(a)=\varepsilon\alpha(a\vee\ker_\varepsilon)=\varepsilon\alpha(\varepsilon(a)\vee\ker_\varepsilon)=\varepsilon\alpha\varepsilon(a)\]
	for all $a\in\mathcal{L}$. Thus $\varepsilon\alpha=\varepsilon\alpha\varepsilon=\alpha\varepsilon$.
\end{proof}

Recall that a \emph{semiring} $S$ is a set with two operations $+$ and $\cdot$ such that $(S,+)$ is a commutative monoid and $(S,\cdot)$ is a monoid such that $a(b+c)=ab+ac$ and $(b+c)a=ba+ca$ for all $a,b,c\in S$.

\begin{cor}\label{corCbool}
	Let $\mathcal{L}$ be a complete modular lattice and $\mathfrak{m}$ be a submonoid of $\End_{lin}(\mathcal{L})$ closed under complements. Suppose $(\mathfrak{m},+,\circ)$ is a semiring and $\mathcal{L}$ is $\mathfrak{m}$-endoregular. If $C(\mathcal{L})$ is Boolean, then $\mathcal{L}$ is $\mathfrak{m}$-abelian.
\end{cor}

\begin{proof}
	By Proposition \ref{Cboolidemcomm} any two idempotents in $\mathfrak{m}$ commute. Let $\varepsilon^2=\varepsilon\in\mathfrak{m}$ and $\psi\in\pi_{\ker_\varepsilon}\mathfrak{m}\varepsilon$. Since $\mathfrak{m}$ is a semiring, $\varepsilon+\psi\in\mathfrak{m}$ is an idempotent. Therefore $\psi+\varepsilon=\psi\varepsilon+\varepsilon=(\psi+\varepsilon)\varepsilon=\varepsilon(\psi+\varepsilon)=\varepsilon\psi+\varepsilon=\varepsilon$, because $\varepsilon\psi=0$. This implies that $\psi=\psi\varepsilon=\varepsilon\psi=0$. Thus $\pi_{\ker_\varepsilon}\mathfrak{m}\varepsilon=0$. By lemma \ref{pife0central}, $\varepsilon$ is central in $\mathfrak{m}$.
\end{proof}

\begin{lemma}\label{abendocompiq}
	Let $\mathcal{L}$ be a complete modular lattice and $\mathfrak{m}$ be a submonoid of $\End_{lin}(\mathcal{L})$. Suppose that $\mathcal{L}$ is  $\mathfrak{m}$-abelian-endoregular and that there exists a linear isomorphism $\theta:[\mathbf{0},a]\to [\mathbf{0},b]$ with $a,b\in C(\mathcal{L})$. If $\iota_b\theta\pi_a$ and $\iota_a\theta^{-1}\pi_b$ are in $\mathfrak{m}$, then $a=b$.
\end{lemma}

\begin{proof}
	Since $\mathcal{L}$ is  $\mathfrak{m}$-abelian-endoregular,
	\[\theta(a\wedge b)=\theta\pi_a\pi_b(\mathbf{1})=\pi_b\theta\pi_a(\mathbf{1})=\pi_b\theta(a)=\pi_b(b)=b=\theta(a).\]
	It follows that $a\wedge b=a$ because $\theta$ is an isomorphism. Thus $a\leq b$. Analogously, using the isomorphism $\theta^{-1}$, we get that $b\leq a$. Therefore, $a=b$.
\end{proof}

\begin{prop}\label{xyig}
	Let $\mathcal{L}$ be a complete modular lattice and $\mathfrak{m}$ be a submonoid of $\End_{lin}(\mathcal{L})$. Suppose that $\mathcal{L}$ is $\mathfrak{m}$-endoregular. The following conditions are equivalent:
	\begin{enumerate}[label=\emph{(\alph*)}]
		\item $\mathcal{L}$ is $\mathfrak{m}$-abelian.
		\item If $x,y\in\mathcal{L}$ are $\mathfrak{m}$-$\mathcal{L}$-generated and there exists a linear isomorphism $\theta:[\mathbf{0},x]\to [\mathbf{0},y]$, then $x=y$.
		\item If $x,y\in\mathcal{L}$ are $\mathfrak{m}$-$\mathcal{L}$-generated and $x\wedge y=\mathbf{0}$, then there are no nonzero linear morphisms from $[\mathbf{0},x]$ to $[\mathbf{0},y]$.
	\end{enumerate}
\end{prop}

\begin{proof}
	(a)$\Rightarrow$(b) Since $x$ is $\mathfrak{m}$-$\mathcal{L}$-generated, there exists a nonzero linear morphism $\varphi:\mathcal{L}\to[\mathbf{0},x]$ with $\iota_x\varphi\in\mathfrak{m}$. Let $a=\varphi(\mathbf{1})$ and $b=\theta(a)$. Then $\theta|:[\mathbf{0},a]\to[\mathbf{0},b]$ is a linear isomorphism. Since $\mathcal{L}$ is $\mathfrak{m}$-endoregular, $a\in C(\mathcal{L})$ and hence $b\in C(\mathcal{L})$. It follows from Lemma \ref{abendocompiq}, that $\varphi(\mathbf{1})=a=b=\theta(a)\leq y$. Since $x$ is $\mathfrak{m}$-$\mathcal{L}$-generated, $x\leq y$. Analogously, $y\leq x$.
	
	(b)$\Rightarrow$(c) Suppose that $\varphi:[\mathbf{0},x]\to[\mathbf{0},y]$ is a nonzero linear morphism. Since $x$ is $\mathfrak{m}$-$\mathcal{L}$-generated, there exists a linear morphism $\psi:\mathcal{L}\to[\mathbf{0},x]$ with $\iota_x\psi\in\mathfrak{m}$. Since $\mathcal{L}$ is $\mathfrak{m}$-endoregular, $\psi(\mathbf{1})\in C(\mathcal{L})$ and $\psi(\mathbf{1})\wedge\ker_\varphi$ is a complement in $[\mathbf{0},\psi(\mathbf{1})]$. In fact, $\mathbf{0}=(\psi(\mathbf{1})\wedge\ker_\varphi)\wedge z$ and $\psi(\mathbf{1})=(\psi(\mathbf{1})\wedge\ker_\varphi)\vee z$ with $[\mathbf{0},z]\cong[\mathbf{0},\varphi\psi(\mathbf{1})]$. By hypothesis, $z=\varphi\psi(\mathbf{1})\leq x\wedge y=\mathbf{0}$. Thus, $\varphi=0$.
	
	(c)$\Rightarrow$(a) Let $\varepsilon^2=\varepsilon\in\mathfrak{m}$. Then $\mathbf{1}=\ker_\varepsilon\vee\varepsilon(\mathbf{1})$ and $\mathbf{0}=\ker_\varepsilon\wedge\varepsilon(\mathbf{1})$. By hypothesis, there are no nonzero linear morphisms from $[\mathbf{0},\varepsilon(\mathbf{1})]$ to $[\mathbf{0},\ker_\varepsilon]$. It follows that $\pi_{\ker_\varepsilon}\varphi\varepsilon=0$ for all $\varphi\in\mathfrak{m}$. By Lemma \ref{pife0central}, $\varepsilon$ is central in $\mathfrak{m}$.
\end{proof}

\begin{cor}\label{hopfandcohopf}
	Let $\mathcal{L}$ be a complete modular lattice and $\mathfrak{m}$ be a submonoid of $\End_{lin}(\mathcal{L})$. Suppose that $\mathcal{L}$ is $\mathfrak{m}$-abelian-endoregular. The following conditions are equivalent for $\varphi\in\mathfrak{m}$:
	\begin{enumerate}[label=\emph{(\alph*)}]
		\item $\varphi$ is injective.
		\item $\varphi$ is an isomorphism.
		\item $\varphi$ is surjective.
	\end{enumerate}	
\end{cor}

\begin{proof}
	(a)$\Leftrightarrow$(b) Suppose that $\varphi$ is injective. Then, $\varphi:[\mathbf{0},\mathbf{1}]\to[\mathbf{0},\varphi(\mathbf{1})]$ is a linear isomorphism. By Proposition \ref{xyig}, $\mathbf{1}=\varphi(\mathbf{1})$. The converse is obvious. 
	
	(b)$\Leftrightarrow$(c) It follows from Proposition \ref{ker+img}.
\end{proof}

\begin{defn}
    A bounded lattice $\mathcal{L}$ is called \emph{Hopfian} (resp. \emph{cohopfian}) if every injective (resp. surjective) linear endomorphism $\varphi\in\End_{lin}(\mathcal{L})$ is an isomorphism. 
\end{defn}

\begin{cor}[{\cite[Remark 2.23(ii)]{leemodules}}]
	Every abelian endoregular module is Hopfian and cohopfian.
\end{cor}

\begin{proof}
	Suppose that $f:M\to M$ is a monomorphism. By Corollary \ref{hopfandcohopf} , $f_\ast:\Lambda(M)\to \Lambda(M)$ is an isomorphism. Let $x\in M$ and consider $Rx\leq M$. Hence there exists $N\leq M$ such that $f_\ast(N)=Rx$, that is, $f(N)=Rx$. Therefore, there is $n\in N$ such that $f(n)=x$. This implies that $f$ is surjective and so $f$ is an isomorphism. The other condition is analogous.
\end{proof}

For a complete lattice $\mathcal{L}$, its \emph{radical} is defined as $\Rad(\mathcal{L})=\bigwedge\{c\in\mathcal{L}\mid c\;\text{coatom}\}$. In a dual way, its \emph{Socle} is defined as $\Soc(\mathcal{L})=\bigvee\{a\in\mathcal{L}\mid a\;\text{atom}\}$.

\begin{cor}
	Let $\mathcal{L}$ be a complete modular lattice with $\Rad(\mathcal{L})\neq\mathcal{L}$. If $\mathcal{L}$ is abelian endoregular, then $\mathcal{L}$ has at most one atom. Moreover, if $\mathcal{L}$ has an atom $a$, then there exists $c\in\mathcal{L}$ such that $a\wedge c=\mathbf{0}$, $a\vee c=\mathbf{1}$, and $\Soc([\mathbf{0},c])=\mathbf{0}$.
\end{cor}

\begin{proof}
	If $\mathcal{L}$ has an atom $a$, then there is a coatom $c\in\mathcal{L}$ and a linear morphism $\varphi:\mathcal{L}\to\mathcal{L}$ such that $\ker_{\varphi}=c$ and $\varphi(\mathbf{1})=a$. This implies that every coatom in $\mathcal{L}$ is $\mathcal{L}$-generated. By Proposition \ref{xyig}, there is at most one atom. Now suppose that there is one atom $a\in\mathcal{L}$. By the above, there exists a coatom $c\in\mathcal{L}$ such that $\mathbf{1}=c\vee a$ and $\mathbf{0}=c\wedge a$. Since $a$ is the unique atom in $\mathcal{L}$, then $\Soc([\mathbf{0},c])=\mathbf{0}$.
	%and $\mathcal{L}\cong[\mathbf{0},a]\times[\mathbf{0},c]$.
\end{proof}

\begin{cor}
	Let $\mathcal{L}$ be a finite complete modular lattice. Then $\mathcal{L}$ is abelian endoregular if and only if $\mathcal{L}\cong 2$.
\end{cor}

\begin{lemma}\label{compdecomp}
	Let $\mathcal{L}$ be an upper-continuous complete modular lattice and $\{a_i\}_I$ be an independent family of elements of $\mathcal{L}$ such that $\bigvee_{i\in I}a_i=\mathbf{1}$. Suppose that each interval $[\mathbf{0},a_i]$ has a decomposition $a_i=b_i\vee c_i$ and $b_i\wedge c_i=\mathbf{0}$. Then $\bigvee_{i\in I}b_i$ is a complement of $\bigvee_{i\in I}c_i$ in $\mathcal{L}$.
\end{lemma}

\begin{proof}
	It is clear that $\left( \bigvee_{i\in I} b_i\right)\vee\left( \bigvee_{i\in I}c_i\right) =\mathbf{1}$. On the other hand, since each $c_i\leq a_i$, the set $\{c_i\}_I$ is independent in $\mathcal{L}$.	By modularity, 
	\[a_j\wedge\left( \bigvee_{i\in I}c_i\right)= c_j\vee \left(a_j\wedge \bigvee_{i\neq j}c_i \right)\leq c_j\vee \left(a_j\wedge \bigvee_{i\neq j}a_i \right)=c_j.\]
	
	Then
	\[b_j\wedge\left( \bigvee_{i\in I}c_i\right)=b_j\wedge a_j\wedge\left( \bigvee_{i\in I}c_i\right)\leq b_j\wedge c_j=\mathbf{0}\]
	for all $j\in I$. It follows from \cite[Proposition 6.1]{calugareanu2013lattice} that $\{b_j\}\cup\{c_i\}_I$ is independent in $\mathcal{L}$ for each $j\in I$. Let $\{b_1,...,b_k\}$ be a finite subset of $\{b_i\}_I$. We have that $\{b_1\}\cup\{c_i\}_I$ is independent in $\mathcal{L}$. Now,
	\[b_2\wedge\left( b_1\vee\bigvee_{i\in I}c_i\right)=b_2\wedge\left( a_1\vee\bigvee_{i\neq 1}c_i\right)=(b_2\wedge a_2)\wedge\left( a_1\vee\bigvee_{i\neq 1}c_i\right)\]
	\[=b_2\wedge \left( c_2\vee \left( a_2\wedge \left( a_1\vee\bigvee_{i\neq 1,2}c_i\right)\right)\right)= b_2\wedge c_2=\mathbf{0} \]
	because
	\[a_2\wedge \left( a_1\vee\bigvee_{i\neq 1,2}c_i\right)\leq a_2\wedge \left( a_1\vee\bigvee_{i\neq 1,2}a_i\right)=\mathbf{0}.\]
	Thus $\{b_2,b_1\}\cup\{c_i\}_I$ is independent in $\mathcal{L}$. By induction $\{b_1,...,b_k\}\cup\{c_i\}_I$ is independent in $\mathcal{L}$. It follows that $\left( \bigvee_{i\in F}b_i\right) \wedge \left( \bigvee_{i\in I}c_i\right) =\mathbf{0}$ for every finite subset $F\subseteq I$ by \cite[Lemma 6.2]{calugareanu2013lattice}. 
	
	Let $X$ be the set of finite joins of elements of $\{b_i\}_I$, that is, $x\in X$ if and only if $x=b_{i_1}\vee\cdots\vee b_{i_k}$. Then $X$ is directed and $\bigvee X=\bigvee_{i\in I} b_i$. Note that $x\wedge\left( \bigvee_{i\in I}c_i\right)=\mathbf{0}$ for all $x\in X$. Since $\mathcal{L}$ is upper-continuous, 
	\[\left( \bigvee_{i\in I} b_i\right)\wedge\left( \bigvee_{i\in I}c_i\right) =\left( \bigvee X\right) \wedge\left( \bigvee_{i\in I}c_i\right)=\bigvee_{x\in X}\left( x\wedge\left( \bigvee_{i\in I}c_i\right)\right) =\mathbf{0}.\]
\end{proof}

Recall that an element $a\in\mathcal{L}$ is \emph{fully invariant} if $\varphi(a)\leq a$ for all $\varphi\in\End_{lin}(\mathcal{L})$.

\begin{prop}\label{compendo}
	Suppose $\mathcal{L}$ is an upper-continuous complete modular lattice. Let $\{a_i\}_I$ be an independent family of fully invariant elements of $\mathcal{L}$ such that $\bigvee_{i\in I}a_i=\mathbf{1}$. The following conditions are equivalent:
	\begin{enumerate}[label=\emph{(\alph*)}]
		\item $\mathcal{L}$ is endoregular.
		\item $[\mathbf{0},a_i]$ is endoregular.
	\end{enumerate}
\end{prop}

\begin{proof}
	(a)$\Rightarrow$(b) It follows from \cite[Proposition 3.13 and Proposition 3.14]{medina2022mbaer} and Theorem \ref{kerimgsumm}.
	
	(b)$\Rightarrow$(a) Let $\varphi:\mathcal{L}\to\mathcal{L}$ be a linear morphism. For any $i\in I$, $\varphi(a_i)\leq a_i$. This implies that $\varphi|:[\mathbf{0},a_i]\to[\mathbf{0},a_i]$ is a linear morphism. Therefore, $\ker_\varphi\wedge a_i$ is a complement in $[\mathbf{0},a_i]$. On the other hand,
	\[\varphi(\pi_{a_i}(\ker_\varphi))=\varphi\left(\left(\ker_\varphi\vee \bigvee_{i\neq j}a_j\right)\wedge a_i \right)\leq \varphi\left(\ker_\varphi\vee \bigvee_{i\neq j}a_j\right)=\varphi\left(\bigvee_{i\neq j}a_j\right).\]
	Hence,
	\[\varphi(\pi_{a_i}(\ker_\varphi))\leq \varphi\left(\bigvee_{i\neq j}a_j\right)\wedge\varphi(a_i)\leq\left(\bigvee_{i\neq j}a_j\right)\wedge a_i=\mathbf{0}.\]
	Thus, $\pi_{a_i}(\ker_\varphi)\leq\ker_\varphi\wedge a_i$ for all $i\in I$. This implies that $\ker_\varphi=\bigvee_{i\in I}\ker_\varphi\wedge a_i$. Then, $\ker_\varphi$ is a complement in $[\mathbf{0},\bigvee_{i=1}^na_i]=\mathcal{L}$ by Lemma \ref{compdecomp}. On the other hand, $\varphi(a_i)$ is complemented in $[\mathbf{0},a_i]$ for all $i\in I$. Therefore $\varphi(\mathbf{1})=\varphi\left( \bigvee_{i\in I}a_i\right)=\bigvee_{i\in I}\varphi(a_i)$ is complemented in $\mathcal{L}$ by Lemma \ref{compdecomp}. It follows from Theorem \ref{kerimgsumm} that $\mathcal{L}$ is endoregular.
\end{proof}

\begin{cor}
	Suppose $\mathcal{L}$ is an upper-continuous complete modular lattice. Let $\{a_i\}_I$ be an independent family of elements of $\mathcal{L}$ such that $\bigvee_{i\in I}a_i=\mathbf{1}$. The following conditions are equivalent:
	\begin{enumerate}[label=\emph{(\alph*)}]
		\item $\mathcal{L}$ is abelian endoregular.
		\item $[\mathbf{0},a_i]$ is abelian endoregular and $a_i$ is fully invariant in $\mathcal{L}$ for all $i\in I$.
	\end{enumerate}
\end{cor}

\begin{proof}
	(a)$\Rightarrow$(b) It follows from Corollary \ref{compabend} and Proposition \ref{abendofi}.
	
	(b)$\Rightarrow$(a) Given $\varphi:\mathcal{L}\to\mathcal{L}$, we have a linear morphism $\varphi_i=\varphi|:[\mathbf{0},a_i]\to[\mathbf{0},a_i]$ with $\ker_{\varphi_i}=\ker_{\varphi}\wedge a_i$ and $\varphi_i(a_i)=\varphi(a_i)$ for all $i\in I$. By Proposition \ref{ker+img}, $a_i=\ker_{\varphi_i}\vee\varphi_i(a_i)=(\ker_{\varphi}\wedge a_i)\vee\varphi(a_i)$ and $\mathbf{0}=\ker_{\varphi_i}\wedge\varphi_i(a_i)=(\ker_{\varphi}\wedge a_i)\wedge\varphi(a_i)$.
	It follows from Lemma \ref{compdecomp} that $\ker_{\varphi}=\bigvee_{i\in I}\ker_{\varphi}\wedge a_i$ is a complement of $\varphi(\mathbf{1})=\bigvee_{i\in I}\varphi(a_i)$ in $\mathcal{L}$. Thus $\mathbf{1}=\ker_{\varphi}\vee\varphi(\mathbf{1})$ and $\mathbf{0}=\ker_{\varphi}\wedge\varphi(\mathbf{0})$. By Proposition \ref{ker+img}, $\mathcal{L}$ is abelian endoregular.
\end{proof}

\section{Regular quotient monoids of linear endomorphisms}\label{rqm}

Recall that an element $x$ in a bounded lattice $\mathcal{L}$ is \emph{superfluous}, if whenever $x\vee y=\mathbf{1}$ then $y=\mathbf{1}$. Equivalently, $x$ is superfluous in $\mathcal{L}$ if $x$ is essential in $\mathcal{L}^{op}$.

\begin{defn}\label{defkextlif}
	Let $\mathcal{L}$ be a complete lattice and let $\mathfrak{m}$ be a submonoid with zero of $\End_{lin}(\mathcal{L})$. The lattice \begin{itemize}
		\item $\mathcal{L}$ is called \emph{$\mathfrak{m}$-$\mathcal{K}$-extending} if for every $\varphi\in\mathfrak{m}$, there exists $c\in C(\mathcal{L})$ such that $\ker_{\varphi}$ is essential in $[\mathbf{0},c]$.
		
		\item $\mathcal{L}$ is called \emph{$\mathfrak{m}$-$\mathcal{T}$-lifting} if for every $\varphi\in\mathfrak{m}$, there exists $c\in C(\mathcal{L})$ with complement $c'$ such that $c\leq\varphi(\mathbf{1})$ and $\varphi(\mathbf{1})\wedge c'$ superfluous in $[\mathbf{0},c']$.
	\end{itemize}
	If the submonoid we are considering is $\End_{lin}(\mathcal{L})$, we will omit the $\mathfrak{m}$.
\end{defn} 

The notions in Definition \ref{defkextlif} are related to those given in \cite[Definition 4.2]{medina2022mbaer} as the following results show.

\begin{prop}\label{rickex}
	Let $\mathcal{L}$ be a complete modular lattice and let $\mathfrak{m}$ be a submonoid with zero of $\End_{lin}(\mathcal{L})$ containing all the projections. The following conditions are equivalent:
	\begin{enumerate}[label=\emph{(\alph*)}]
		\item $\mathcal{L}$ is $\mathfrak{m}$-Rickart. 
		\item $\mathcal{L}$ is $\mathfrak{m}$-$\mathcal{K}$-extending and $\mathfrak{m}$-$\mathcal{K}$-nonsingular.
	\end{enumerate}
\end{prop}

\begin{proof}
	(a)$\Rightarrow$(b). It is clear that every $\mathfrak{m}$-Rickart lattice is $\mathfrak{m}$-$\mathcal{K}$-extending and it is $\mathfrak{m}$-$\mathcal{K}$-nonsingular by \cite[Lemma 4.14]{medina2022mbaer}.
	
	(b)$\Rightarrow$(a). Let $\varphi\in\mathfrak{m}$. By hypothesis, there exists $c\in C(\mathcal{L})$ such that $\ker_{\varphi}$ is essential in $[\mathbf{0},c]$. Let $c'$ be a complement of $c$ and $\psi=\varphi\pi_c\in\mathfrak{m}$. Then $\ker_{\psi}=\ker_\varphi\vee c'$ which is essential in $\mathcal{L}$. This implies that $\psi=0$ and hence $c\leq ker_\varphi$, hence  $c=\ker_{\varphi}$. That is, $\mathcal{L}$ is $\mathfrak{m}$-Rickart.
\end{proof} 

\begin{prop}\label{drictlif}
	Let $\mathcal{L}$ be a complete modular lattice and let $\mathfrak{m}$ be a submonoid with zero of $\End_{lin}(\mathcal{L})$ containing all the projections. The following conditions are equivalent:
	\begin{enumerate}[label=\emph{(\alph*)}]
		\item $\mathcal{L}$ is dual-$\mathfrak{m}$-Rickart.
		\item $\mathcal{L}$ is $\mathfrak{m}$-$\mathcal{T}$-lifting and $\mathfrak{m}$-$\mathcal{T}$-nonsingular.
	\end{enumerate}
\end{prop}

\begin{proof}
	(a)$\Rightarrow$(b) It is clear that every dual-$\mathfrak{m}$-Rickart lattice is $\mathfrak{m}$-$\mathcal{T}$-lifting and it is $\mathfrak{m}$-$\mathcal{T}$-nonsingular by \cite[Lemma 4.16]{medina2022mbaer}.
	
	(b)$\Rightarrow$(a) Let $\varphi\in\mathfrak{m}$. By hypothesis, there exists $c\in C(\mathcal{L})$ with complement $c'$ such that $c\leq\varphi(\mathbf{1})$ and $\varphi(\mathbf{1})\wedge c'$ superfluous in $[\mathbf{0},c']$. We have that $\pi_{c'}\varphi\in\mathfrak{m}$ and $\pi_{c'}(\varphi(\mathbf{1}))=\varphi(\mathbf{1})\wedge c'$. Since $\mathcal{L}$ is $\mathfrak{m}$-$\mathcal{T}$-nonsingular, $\pi_{c'}\varphi=0$, thus $\varphi(\mathbf{1})\leq c$. Hence   $\mathcal{L}$ is dual-$\mathfrak{m}$-Rickart.
\end{proof}
\begin{defn}
	Let $M$ be an $R$-module.
	\begin{itemize}
		\item $M$ is called \emph{$\mathcal{K}$-extending} if $\Lambda(M)$ is $\mathfrak{E}_M$-$\mathcal{K}$-extending.
		\item $M$ is called \emph{$\mathcal{T}$-lifting} if $\Lambda(M)$ is $\mathfrak{E}_M$-$\mathcal{T}$-lifting.
	\end{itemize}
\end{defn}

\begin{rem}
	It follows that an $R$-module $M$ is $\mathcal{K}$-extending if and only if for every $f\in\End_{R}(M)$ there exists a direct summand $N$ of $M$ such that $\Ker f\ess N$. On the other hand, $M$ is $\mathcal{T}$-lifting if and only if for every $f\in\End_{R}(M)$ there exists a decomposition $M=N\oplus L$ such that $N\leq f(M)$ and $f(M)\cap L<< L$.
\end{rem}

\begin{cor}
	The following conditions are equivalent for an $R$-module $M$:
	\begin{enumerate}[label=\emph{(\alph*)}]
		\item $M$ is Rickart.
		\item $M$ is $\mathcal{K}$-extending and $\mathcal{K}$-nonsingular.
	\end{enumerate}
\end{cor}

\begin{cor}
	The following conditions are equivalent for an $R$-module $M$:
	\begin{enumerate}[label=\emph{(\alph*)}]
		\item $M$ is dual-Rickart.
		\item $M$ is $\mathcal{T}$-lifting and $\mathcal{T}$-nonsingular.
	\end{enumerate}
\end{cor}

\begin{defn}
	Let $\mathcal{L}$ be a complete modular lattice and $\mathfrak{m}$ a submonoid of $\End_{lin}(\mathcal{L})$. We define the following relations on $\mathfrak{m}$:
	\[\varphi\equiv_\Delta\psi\;\Leftrightarrow\;\text{ there exists }x\text{ essential in }\mathcal{L}\;\text{ such that }\varphi(y)=\psi(y)\;\text{ for all }y\leq x.\]
	\[\varphi\equiv^\nabla \psi\;\Leftrightarrow\;\text{ there exists }x\text{ superfluous in }\mathcal{L}\;\text{ such that }\varphi(a)\vee x=\psi(a)\vee x\;\text{ for all }a\in\mathcal{L}.\]
\end{defn}

\begin{lemma}\label{imginvess}
Let $\mathcal{L}$ be a complete modular lattice and $\varphi\in\End_{lin}(\mathcal{L})$. 
\begin{enumerate}
	\item If $x\in\mathcal{L}$ is essential, then $w=\bigvee\{a\in\mathcal{L}\mid \varphi(a)\leq x\}$ is essential in $\mathcal{L}$.
	\item If $x\in\mathcal{L}$ is superfluous, then $\varphi(x)$ is superfluous in $\mathcal{L}$.
\end{enumerate}
\end{lemma}

\begin{proof}
    \textit{1.} Let $w=\bigvee\{a\in\mathcal{L}\mid \varphi(a)\leq x\}$. Note that $\ker_\varphi\leq w$ and $\varphi(w)\leq x\wedge\varphi(\mathbf{1})$. On the other hand, $\overline{\varphi}^{-1}(x\wedge\varphi(\mathbf{1}))\leq w$. This implies that $\overline{\varphi}^{-1}(x\wedge\varphi(\mathbf{1}))=w$. Since $x$ is essential in $\mathcal{L}$, $x\wedge\varphi(\mathbf{1})$ is essential in $[\mathbf{0},\varphi(\mathbf{1})]$. This implies that $w$ is essential in $[\ker_\varphi,\mathbf{1}]$. Let $y\in\mathcal{L}$ such that $w\wedge y=\mathbf{0}$. Then $w\wedge(y\vee\ker_\varphi)=\ker_\varphi$. Hence $y\vee\ker_\varphi=\ker_\varphi$, that is, $y\leq\ker_\varphi$. Therefore, $y\leq w$ and so $y=\mathbf{0}$. Thus, $w$ is essential in $\mathcal{L}$.
    
    \textit{2}. Since $x$ is superfluous in $\mathcal{L}$, $x\vee\ker_{\varphi}$ is superfluous in $[\ker_{\varphi},\mathbf{1}]$. This implies that $\varphi(x)$ is superfluous in $[\mathbf{0},\varphi(\mathbf{1})]$. Let $y\in\mathcal{L}$ such that $\varphi(x)\vee y=\mathbf{1}$. Then $\varphi(x)\vee(\varphi(\mathbf{1})\wedge y)=\varphi(\mathbf{1})\wedge(\varphi(x)\vee y)=\varphi(\mathbf{1})$. This implies that $\varphi(\mathbf{1})\wedge y=\varphi(\mathbf{1})$, that is, $\varphi(\mathbf{1})\leq y$. Therefore, $\mathbf{1}=\varphi(x)\vee y=y$. Thus $\varphi(x)$ is superfluous in $\mathcal{L}$.
\end{proof}

\begin{lemma}\label{congru}
	Let $\mathcal{L}$ be a complete modular lattice. The relations $\equiv_\Delta$ and $\equiv^\nabla$ are congruences on any submonoid $\mathfrak{m}\subseteq\End_{lin}(\mathcal{L})$.
\end{lemma}

\begin{proof}
	$(\equiv_\Delta)$. Let $\varphi,\psi,\sigma\in\mathfrak{m}$. We have that $\varphi\equiv_\Delta\varphi$ because $\varphi(y)=\varphi(y)$ for all $y\leq\mathbf{1}$ and $\mathbf{1}$ is essential in $\mathcal{L}$. It is clear that $\equiv_\Delta$ is symmetric. Now, suppose that $\varphi\equiv_\Delta\psi$ and $\psi\equiv_\Delta\sigma$. Then there exist $x$ and $w$ essential in $\mathcal{L}$ such that $\varphi(y)=\psi(y)$ for all $y\leq x$ and $\psi(v)=\sigma(v)$ for all $v\leq w$. Since $x$ and $w$ are essential in $\mathcal{L}$, so is $x\wedge w$. Let $z\leq x\wedge w$. Then $z\leq x$ and $z\leq w$. Therefore, $\varphi(z)=\psi(z)=\sigma(z)$. Thus $\varphi\equiv_\Delta\sigma$. This proves that $\equiv_\Delta$ is an equivalence relation.
	
	Now, suppose that $\varphi\equiv_\Delta\psi$. Then there exists $x$ essential in $\mathcal{L}$ such that $\varphi(y)=\psi(y)$ for all $y\leq x$. Let $y\leq x$, then $\sigma\varphi(y)=\sigma\psi(y)$. Therefore $\sigma\varphi\equiv_\Delta\sigma\psi$. Let $w=\bigvee\{a\in\mathcal{L}\mid \sigma(a)\leq x\}$. By Lemma \ref{imginvess}, $w$ is essential in $\mathcal{L}$. Note that $\sigma(w)\leq x$. Let $v\leq w$. Then $\sigma(v)\leq\sigma(w)\leq x$. Hence $\varphi(\sigma(v))=\psi(\sigma(v))$. Thus $\varphi\sigma\equiv_\Delta\psi\sigma$.
	
	$(\equiv^\nabla)$. Let $\varphi,\psi,\sigma\in\mathfrak{m}$. We have that $\varphi\equiv^\nabla\varphi$ because $\varphi(a)\vee\mathbf{0}=\varphi(a)\vee\mathbf{0}$ for all $a\in\mathcal{L}$. It is clear that $\equiv^\nabla$ is symmetric. Now, suppose that $\varphi\equiv^\nabla\psi$ and $\psi\equiv^\nabla\sigma$. Then there exist $x$ and $y$ superfluous in $\mathcal{L}$ such that $\varphi(a)\vee x=\psi(a)\vee x$ and $\psi(a)\vee y=\sigma(a)\vee y$ for all $a\in\mathcal{L}$. Hence $\varphi(a)\vee x\vee y=\sigma(a)\vee x\vee y$ for all $a\in\mathcal{L}$ and $x\vee y$ is superfluous in $\mathcal{L}$. Thus $\varphi\equiv^\nabla\sigma$. This proves that $\equiv^\nabla$ is an equivalence relation. 
	
	Now, suppose that $\varphi\equiv^\nabla\psi$. Then there exists $x$ superfluous in $\mathcal{L}$ such that $\varphi(a)\vee x=\psi(a)\vee x$ for all $a\in \mathcal{L}$. It follows that $\varphi(\sigma(a))\vee x=\psi(\sigma(a))\vee x$ for all $a\in\mathcal{L}$. Thus $\varphi\sigma\equiv^\nabla\psi\sigma$. On the other hand,
	\[\sigma(\varphi(a))\vee\sigma(x)=\sigma(\varphi(a)\vee x)=\sigma(\psi(a)\vee x)=\sigma(\psi(a))\vee\sigma(x)\]
	for all $a\in\mathcal{L}$,  and $\sigma(x)$ is superfluous in $\mathcal{L}$ by Lemma \ref{imginvess}. Thus $\sigma\varphi\equiv^\nabla\sigma\psi$.
\end{proof}

\begin{lemma}
	Let $\mathcal{L}$ be a complete modular lattice. Then,
	\begin{enumerate}
		\item $0\equiv_\Delta\varphi$ if and only if $\ker_{\varphi}$ is essential in $\mathcal{L}$.
		\item $0\equiv^\nabla\varphi$ if and only if $\varphi(\mathbf{1})$ is superfluous in $\mathcal{L}$.
	\end{enumerate}
\end{lemma}

\begin{proof}
	\textit{1.} $\Rightarrow$ Let $\varphi\in\End_{lin}(\mathcal{L})$ such that $0\equiv_\Delta \varphi$. Then, there exists $x$ essential in $\mathcal{L}$ such that $\mathbf{0}=0(y)=\varphi(y)$ for all $y\leq x$. This implies that $x\leq\ker_{\varphi}$ and, therefore $\ker_{\varphi}$ is essential in $\mathcal{L}$.
	
	$\Leftarrow$ It is clear that if $\ker_{\varphi}$ is essential in $\mathcal{L}$, then $0\equiv_\Delta\varphi$.
	
	\textit{2}. $\Rightarrow$ Let $\varphi\in\End_{lin}(\mathcal{L})$ such that $0\equiv^\nabla \varphi$. Then, there exists $x$ superfluous in $\mathcal{L}$ such that $\varphi(a)\vee x=0(a)\vee x=\mathbf{0}\vee x=x$ for all $a\in \mathcal{L}$. Hence $\varphi(\mathbf{1})\leq x$. Thus, $\varphi(\mathbf{1})$ is superfluous in $\mathcal{L}$.
	
	$\Leftarrow$ We have that $\varphi(a)\vee\varphi(\mathbf{1})=0(a)\vee\varphi(\mathbf{1})$ for all $a\in\mathcal{L}$ and $\varphi(\mathbf{1})$ superfluous in $\mathcal{L}$. Thus $\varphi\equiv^\nabla 0$.
\end{proof}

Let $\mathcal{L}$ be a complete modular lattice and let $[\varphi]_\Delta$ and $[\varphi]^\nabla$ denote the equivalence classes of $\varphi\in\End_{lin}(\mathcal{L})$ respect to $\equiv_\Delta$ and to $\equiv^\nabla$,  respectively. Let $\Delta$ denote  
\[[0]_\Delta=\{\varphi\in\End_{lin}(\mathcal{L})\mid 0\equiv_\Delta \varphi\}=\{\varphi\in\End_{lin}(\mathcal{L})\mid \ker_{\varphi}\;\text{is essential in }\mathcal{L}\},\]
and let $\nabla$ denote
\[[0]^\nabla=\{\varphi\in\End_{lin}(\mathcal{L})\mid 0\equiv^\nabla \varphi\}=\{\varphi\in\End_{lin}(\mathcal{L})\mid \varphi(\mathbf{1})\text{ is superfluous in }\mathcal{L}\}.\] 
The proof of the following lemma is straightforward and we omit the proof.

\begin{lemma}
	Let $\mathcal{L}$ be a complete modular lattice. Then,
	\begin{enumerate}
		\item $\Delta$ and $\nabla$ are ideals  of $\End_{lin}(\mathcal{L})$.
		\item $\Delta$ and $\nabla$ contain no nonzero idempotents.
	\end{enumerate}
\end{lemma}

\begin{prop}\label{equivmod}
	Let $M$ be an $R$-module and $f,g \in\End_{R}(M)$. Then,
	\begin{enumerate}
		\item If $\Ker(f-g)$ is essential in $M$ then $f_\ast\equiv_\Delta g_\ast$.
		\item If $\Img(f-g)$ is superfluous in $M$ then $f_\ast\equiv^\nabla g_\ast$.
	\end{enumerate}
\end{prop}

\begin{proof}
	\textit{(1)} Let $N\leq \Ker(f-g)$. Then $f(n)=g(n)$ for all $n\in N$. In particular, $f_\ast(N)=f(N)=g(N)=g_\ast(N)$. Thus, $f_\ast\equiv_\Delta g_\ast$.
	
	\textit{(2)} Let $N\leq M$ and $n\in N$. Consider $f(n)+(f-g)(m)\in f(N)+\Img(f-g)$. Then 
	\[f(n)+(f-g)(m)=f(n)+(f-g)(m)+g(n)-g(n)=g(n)+(f-g)(n)+(f-g)(m)\]
	\[=g(n)+(f-g)(n+m)\in g(N)+\Img(f-g).\]
	Hence $f(N)+\Img(f-g)\subseteq g(N)+\Img(f-g)$. Analogously, $g(N)+\Img(f-g)\subseteq f(N)+\Img(f-g)$. Thus $g_\ast(N)+\Img(f-g)=f_\ast(N)+\Img(f-g)$ for all $N\leq M$, that is, $f_\ast\equiv^\nabla g_\ast$.
\end{proof}

\begin{lemma}
	Let $\mathcal{L}$ be a complete modular lattice and $\mathfrak{m}$ a submonoid of $\End_{lin}(\mathcal{L})$. 
	\begin{enumerate}
		\item If $\mathcal{L}$ satisfies the $\mathfrak{m}$-$C_2$ condition, then every monomorphism  $\varphi\in\mathfrak{m}$ such that $\varphi(\mathbf{1})$ is essential in $\mathcal{L}$, is an isomorphism, and $[1]_\Delta$ consists of isomorphisms.
		
		\item If $\mathcal{L}$ satisfies the $\mathfrak{m}$-$D_2$ condition, then every $\varphi\in\mathfrak{m}$ surjective such that $\ker_{\varphi}$ is superfluous in $\mathcal{L}$, is an isomorphism, and $[1]^\nabla$ consists of isomorphisms.
	\end{enumerate}
\end{lemma}

\begin{proof}
	\textit{1.} Let $\varphi\in\mathfrak{m}$ be a monomorphism such that $\varphi(\mathbf{1})$ is essential in $\mathcal{L}$. By $\mathfrak{m}$-$C_2$, $\varphi(\mathbf{1})$ is a complement in $\mathcal{L}$. It follows that $\varphi(\mathbf{1})=\mathbf{1}$. Thus, $\varphi$ is a linear isomorphism. Let $\varphi\in[1]_\Delta$. Then there exists $x$ essential in $\mathcal{L}$, such that $\varphi(y)=y$ for all $y\leq x$. Suppose that $z\in\mathcal{L}$ is such that $\varphi(z)=\mathbf{0}$. Then $z\wedge x=\varphi(x\wedge z)=\mathbf{0}$. This implies that $z=\mathbf{0}$. Therefore, $\varphi$ is a monomorphism. Since $x=\varphi(x)\leq\varphi(\mathbf{1})$, $\varphi(\mathbf{1})$ is essential in $\mathcal{L}$. It follows that $\varphi$ is an isomorphism.
	
	\textit{2.} Let $\varphi\in\mathfrak{m}$ be surjective such that $\ker_\varphi$ is superfluous in $\mathcal{L}$. Then there is an isomorphism $\overline{\varphi}:[\ker_{\varphi},\mathbf{1}]\to[\mathbf{0},\mathbf{1}]$. It follows from the condition $\mathfrak{m}$-$D_2$ that $\ker_\varphi$ is a complement. The hypothesis implies that $\ker_{\varphi}=\mathbf{0}$. Thus $\varphi$ is an isomorphism. Let $\varphi\in[1]^\nabla$. Then there exists $x$ superfluous in $\mathcal{L}$ such that $\varphi(a)\vee x=1(a)\vee x=a\vee x$ for all $a\in \mathcal{L}$. In particular, $\varphi(\mathbf{1})\vee \mathbf{1}=\mathbf{1}\vee x=\mathbf{1}$. Since $x$ is superfluous, $\varphi(\mathbf{1})=\mathbf{1}$, that is, $\varphi$ is surjective. On the other hand $x=\mathbf{0}\vee x=\varphi(\ker_{\varphi})\vee x=\ker_\varphi\vee x$. This implies that $\ker_\varphi\leq x$. Therefore, $\ker_{\varphi}$ is superfluous in $\mathcal{L}$. Thus, $\varphi$ is an isomorphism.
\end{proof}

\begin{thm}\label{thmDelta}
	Let $\mathcal{L}$ be a complete modular lattice and $\mathfrak{m}$ be a submonoid of $\End_{lin}(\mathcal{L})$ closed under complements. If $\mathcal{L}$ is $\mathfrak{m}$-$\mathcal{K}$-extending and satisfies $\mathfrak{m}$-$C_2$, then $\mathfrak{m}/\equiv_\Delta$ is a regular monoid.
\end{thm}

\begin{proof}
	By Lemma \ref{congru}, $\mathfrak{m}/\equiv_\Delta$ is a monoid. Let $\varphi\in\mathfrak{m}$. Since $\mathcal{L}$ is $\mathfrak{m}$-$\mathcal{K}$-extending, there exists $c\in C(\mathcal{L})$ such that $\ker_{\varphi}$ is essential in $[\mathbf{0},c]$. Let $c'$ be a complement of $c$ in $\mathcal{L}$. The morphism $\varphi$ induces a linear isomorphism $\varphi|:[\mathbf{0},c']\to[\mathbf{0},\varphi(c')]$. Since $\varphi\pi_{c'}\in\mathfrak{m}$ and $\mathcal{L}$ satisfies $\mathfrak{m}$-$C_2$, $\varphi(c')$ is a complement in $\mathcal{L}$. By the hypothesis on $\mathfrak{m}$, we can extend $(\varphi|)^{-1}$ to a linear endomorphism $\psi\in\mathfrak{m}$ such that $\psi\varphi(x)=x$ for all $x\leq c'$. Thus $\varphi\psi\varphi(c')=\varphi(c')$. Moreover $\varphi\psi\varphi(\ker_{\varphi}\vee c')=\varphi(\ker_{\varphi}\vee c')$. Since $\ker_{\varphi}$ is essential in $[\mathbf{0},c]$ and $c'$ is a complement of $c$, $\ker_{\varphi}\vee c'$ is essential in $\mathcal{L}$. Let $y\leq\ker_{\varphi}\vee c'$. Then $\varphi(y)\leq\varphi(\ker_{\varphi}\vee c')=\varphi(c')$. This implies that there exists $\mathbf{0}\leq z\leq c'$ such that $\varphi(z)=\varphi(y)$. Therefore $\varphi\psi\varphi(y)=\varphi\psi\varphi(z)=\varphi(z)=\varphi(y)$. Thus, $\varphi\psi\varphi\equiv_\Delta\varphi$. 
\end{proof}

Since every endoregular lattice $\mathcal{L}$ is $\mathcal{K}$-extending and satisfies $(C_2)$, one might expect that, for those lattices, the congruence $\equiv_\Delta$ in $\End_{lin}(\mathcal{L})$ is trivial, but it is not the case. The following example shows an endoregular lattice $\mathcal{L}$ such that $\equiv_\Delta$ is not trivial.

\begin{example}
	Consider the following lattice $\mathcal{L}$
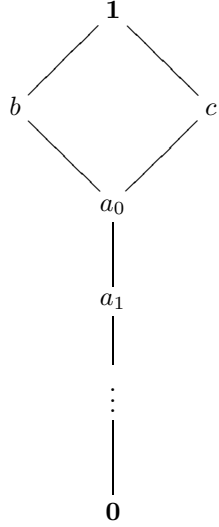
	\[\xymatrix{ & \mathbf{1}\ar@{-}[dl]\ar@{-}[dr] & \\ b\ar@{-}[dr] & & c\ar@{-}[dl] \\ & a_0\ar@{-}[d] & \\ & a_1\ar@{-}[d] & \\ & \vdots\ar@{-}[d] & \\  & \mathbf{0} &  }\]
	
	It is not difficult to see that every nonzero linear endomorphism of $\mathcal{L}$ is an isomorphism. Therefore $\mathcal{L}$ is endoregular. Consider the following endomorphism
	\[\begin{array}{c l}
	\varphi(x)=x & \text{for all } x\leq a_0 \\
	\varphi(b)=c & \\
	\varphi(c)=b & \\
	\varphi(\mathbf{1})=\mathbf{1} 
	\end{array}\]
	Then $\varphi\equiv_\Delta id$. In fact, $\End_{lin}(\mathcal{L})=\{0,id,\varphi\}$ and $\End_{lin}(\mathcal{L})/\equiv_\Delta=\{[0],[id]\}$.
\end{example}

The module theoretic version of Theorem \ref{thmDelta} is stated for continuous modules \cite[Proposition 3.15]{mohamedcontinuous}, that is modules satisfying the conditions $(C_1)$ and $(C_2)$. As an example of Theorem \ref{thmDelta}, we show a lattice that is $\mathcal{K}$-extending but does not satisfy $(C_1)$ \cite[definition 1.1]{albu2016conditions}.

\begin{example}
	Consider the following lattice $\mathcal{L}$
	\[\xymatrix{ & \mathbf{1}\ar@{-}[d] & \\ & c\ar@{-}[dr]\ar@{-}[dl] & \\ a\ar@{-}[dr] & & b\ar@{-}[dl] \\ & \mathbf{0} & }\]
	It can be seen easily that $\mathcal{L}$ does not satisfy $C_1$. There are 5 linear endomorphism on $\mathcal{L}$, $\End_{lin}(\mathcal{L})=\{0,id,\varphi,\psi,\tau\}$ given by 
	\[\begin{array}{c c c}
	\varphi(\mathbf{0})=\mathbf{0} & \psi(\mathbf{0})=\mathbf{0} & \tau(\mathbf{0})=\mathbf{0} \\
	\varphi(a)=\mathbf{0} & \psi(a)=\mathbf{0} & \tau(a)=b \\
	\varphi(b)=\mathbf{0} & \psi(b)=\mathbf{0} & \tau(b)=a \\
	\varphi(c)=\mathbf{0} & \psi(c)=\mathbf{0} & \tau(c)=c \\
	\varphi(\mathbf{1})=a & \psi(\mathbf{1})=b & \tau(\mathbf{1})=\mathbf{1}.
	\end{array}\]
	Hence, $\mathcal{L}$ is $\mathcal{K}$-extending. Then $\End_{lin}(\mathcal{L})/\equiv_\Delta=\{[0],[id],[\tau]\}$ is a regular monoid. %On the other hand, $\mathcal{L}^{op}$ does not satisfy $D_1$ but $\End_{lin}(\mathcal{L})/\equiv^\nabla$ is a regular monoid.
\end{example} 

\begin{cor}
	Let $\mathcal{L}$ be an indecomposable modular lattice, that is, $C(\mathcal{L})=\{\mathbf{0},\mathbf{1}\}$. Consider the following sentences:
	\begin{enumerate}
		\item Every linear endomorphism $\varphi\notin\Delta$ has an inverse.
		\item $\End_{lin}(\mathcal{L})/\equiv_\Delta$ is a monoid in which every nonzero element has an inverse.
		\item $\mathcal{L}$ is $\mathcal{K}$-extending.
	\end{enumerate}
	Then (1)$\Rightarrow$(2)$\Rightarrow$(3). Moreover, if $\mathcal{L}$ is Hopfian, then the three conditions are equivalent.
\end{cor}

\begin{proof}
	\textit{(1)}$\Rightarrow$\textit{(2)} It is clear.
	
	\textit{(2)}$\Rightarrow$\textit{(3)} Let $\varphi\in\End_{lin}(\mathcal{L})$. By hypothesis, $\ker_\varphi$ is essential in $\mathcal{L}$ or there exist $\psi\in\End_{lin}(\mathcal{L})$ and $x\in\mathcal{L}$ essential such that $\psi\varphi(y)=id(y)=y$ for all $y\leq x$. Consider $\ker_\varphi$ and set $y=x\wedge\ker_\varphi$. It follows that $\mathbf{0}=\psi\varphi(y)=y$. Since $x$ is essential, $\ker_\varphi=\mathbf{0}$. Thus, $\mathcal{L}$ is $\mathcal{K}$-extending.
	
	Suppose $\mathcal{L}$ is Hopfian. \textit{(3)}$\Rightarrow$\textit{(1)} Let $\varphi\in\End_{lin}(\mathcal{L})$. If $\ker_{\varphi}\neq\mathbf{0}$, then $\ker_{\varphi}$ is essential in $\mathcal{L}$ by the hypothesis. Therefore, $\varphi\in\Delta$. Now, if $\ker_{\varphi}=\mathbf{0}$ then $\varphi$ is an isomorphism because $\mathcal{L}$ is Hopfian. 
\end{proof}

\begin{thm}\label{thmnabla}
	Let $\mathcal{L}$ be a complete modular lattice and $\mathfrak{m}$ a submonoid of $\End_{lin}(\mathcal{L})$ closed under complements. If $\mathcal{L}$ is $\mathfrak{m}$-$\mathcal{T}$-lifting and satisfies $\mathfrak{m}$-$D_2$, then $\mathfrak{m}/\equiv^\nabla$ is a regular monoid.
\end{thm}

\begin{proof}
	By Lemma \ref{congru}, $\mathfrak{m}/\equiv^\nabla$ is a monoid. Let $\varphi\in\mathfrak{m}$. Since $\mathcal{L}$ is $\mathfrak{m}$-$\mathcal{T}$-lifting, there exists $c\in C(\mathcal{L})$ there exists $c\in C(\mathcal{L})$ with complement $c'$ such that $c\leq\varphi(\mathbf{1})$ and $\varphi(\mathbf{1})\wedge c'$ is superfluous in $[\mathbf{0},c']$. Consider the isomorphism $\overline{\varphi}:[\ker_{\varphi},\mathbf{1}]\to[\mathbf{0},\varphi(\mathbf{1})]$. Then, there exist $\ker_{\varphi}\leq x,y\leq \mathbf{1}$ such that $\varphi(x)=c$, $\varphi(y)=\varphi(\mathbf{1})\wedge c'$ and $\overline{\varphi}$ induces isomorphisms $[\ker_{\varphi},x]\cong[\mathbf{0},c]$ and $[\ker_{\varphi},y]\cong[\mathbf{0},\varphi(\mathbf{1})\wedge c']$. Furthermore, $x$ and $y$ are complement one of each other in the interval $[\ker_{\varphi},\mathbf{1}]$. Therefore, $[y,\mathbf{1}]\cong[\ker_{\varphi},x]\cong[\mathbf{0},c]$. Since $\pi_{c}\varphi\in\mathfrak{m}$ and the $\mathfrak{m}$-$D_2$ condition, $y\in C(\mathcal{L})$. Let $z\in\mathcal{L}$ be a complement of $y$. Then $\ker_{\varphi}\wedge z=\mathbf{0}$ and hence $\varphi|:[\mathbf{0},z]\to[\mathbf{0},\varphi(z)]$ is an isomorphism. We claim that $\varphi(z)$ is a complement of $c'$ in $\mathcal{L}$. We have that
	\[\varphi(\mathbf{1})=\varphi(y\vee z)=\varphi(y)\vee\varphi(z)=(\varphi(\mathbf{1})\wedge c')\vee\varphi(z)=\varphi(\mathbf{1})\wedge(c'\vee\varphi(z)).\]
	Thus, $c\leq\varphi(\mathbf{1})\leq c'\vee\varphi(\mathbf{1})$. By modularity, $c'\wedge\_:[c,\mathbf{1}]\to[\mathbf{0},c']$ is an isomorphism, and $c'\wedge(c'\vee\varphi(z))=c'$. This implies that $c'\vee\varphi(z)=\mathbf{1}$. On the other hand, there exists $0\leq w\leq z$ such that $\varphi(w)=c'\wedge\varphi(z)$. Then $\varphi(w\vee\ker_{\varphi})=c'\wedge\varphi(z)\leq\varphi(y)$. This implies that $w\vee\ker_{\varphi}\leq y\wedge(z\vee\ker_{\varphi})=\ker_{\varphi}\vee\mathbf{0}=\ker_{\varphi}$. Hence $w\leq\ker_{\varphi}$. Thus $c'\wedge\varphi(z)=\varphi(w)=\mathbf{0}$. This proves the claim. By the hypothesis on $\mathfrak{m}$, we can consider the endomorphism $\psi\in\mathfrak{m}$ given by $\iota_z(\varphi|)^{-1}\pi_{\varphi(z)}:\mathcal{L}\to\mathcal{L}$. Therefore 
	\[\varphi\psi\varphi(a)=\pi_{\varphi(z)}(\varphi(a))=(\varphi(a)\vee c')\wedge\varphi(z),\]
	for all $a\in\mathcal{L}$. Note that $((\varphi(a)\vee c')\wedge\varphi(z))\vee c'=(\varphi(a)\vee c')\wedge(\varphi(z)\vee c')=(\varphi(a)\vee c')\wedge\mathbf{1}=\varphi(a)\vee c'$, for all $a\in\mathcal{L}$. Hence
	\[\varphi\psi\varphi(a)\vee(\varphi(\mathbf{1})\wedge c')=\varphi(\mathbf{1})\wedge(\varphi\psi\varphi(a)\vee c')=\varphi(\mathbf{1})\wedge(\varphi(a)\vee c')=\varphi(a)\vee(\varphi(\mathbf{1})\wedge c'),\]
	for all $a\in\mathcal{L}$. It is not difficult to see that $\varphi(\mathbf{1})\wedge c'$ is superfluous in $\mathcal{L}$ because $\varphi(\mathbf{1})$ is superfluous in $[c,\mathbf{1}]$ and $c\in C(\mathcal{L})$. Thus $\varphi\psi\varphi\equiv^\nabla \varphi$.
\end{proof}

\begin{cor}
	Let $\mathcal{L}$ be an indecomposable modular lattice, i.e., $C(\mathcal{L})=\{\mathbf{0},\mathbf{1}\}$. Consider the following conditions:
	\begin{enumerate}
		\item Every linear endomorphism $\varphi\notin\nabla$ has an inverse.
		\item $\End_{lin}(\mathcal{L})/\equiv^\nabla$ is a monoid in which every nonzero element has an inverse.
		\item $\mathcal{L}$ is $\mathcal{T}$-lifting.
	\end{enumerate}
	Then (1)$\Rightarrow$(2)$\Rightarrow$(3). Moreover, if $\mathcal{L}$ is cohopfian, then the three conditions are equivalent.
\end{cor}

\begin{proof}
	\textit{(1)}$\Rightarrow$\textit{(2)} It is clear.
	
	\textit{(2)}$\Rightarrow$\textit{(3)} Let $\varphi\in\End_{lin}(\mathcal{L})$. By hypothesis $\varphi(\mathbf{1})$ is superfluous in $\mathcal{L}$ or there exist $\psi\in\End_{lin}(\mathcal{L})$ and $x\in\mathcal{L}$ superfluous such that $\varphi\psi(a)\vee x=a\vee x$ for all $a\in\mathcal{L}$. Then $\varphi\psi(\mathbf{1})\vee x=\mathbf{1}\vee x=\mathbf{1}$. This implies that $\varphi\psi(\mathbf{1})=\mathbf{1}$. Therefore, $\varphi(\mathbf{1})=\mathbf{1}$. Thus, $\mathcal{L}$ is $\mathcal{T}$-lifting.
	
	Suppose $\mathcal{L}$ is cohopfian. \textit{(3)}$\Rightarrow$\textit{(1)} Let $\varphi\in\End_{lin}(\mathcal{L})$. If $\varphi(\mathbf{1})\neq\mathbf{1}$, then $\varphi(\mathbf{1})$ is superfluous in $\mathcal{L}$ by the hypothesis. Therefore, $\varphi\in\nabla$. Now, if $\varphi(\mathbf{1})=\mathbf{1}$ then $\varphi$ is an isomorphism because $\mathcal{L}$ is cohopfian.
\end{proof}	

In \cite[Corollary 2.32]{mohamedcontinuous}, it is proved that in a quasi-continuous module, two isomorphic submodules have isomorphic closures (given by the $(C_1)$ condition). In a dual way, in \cite[Theorem 4.24]{mohamedcontinuous}, it is proved that given two direct summands $A$ and $B$ of a quasi-discrete module $M$, such that $A/X\cong B/Y$ with $X$ superfluous in $A$ and $Y$ superfluous in $B$, then $A\cong B$. We finish this section giving an example which shows that the mentioned results cannot be extended to linear lattices.

\begin{example}
	Consider the following lattice $\mathcal{L}$:
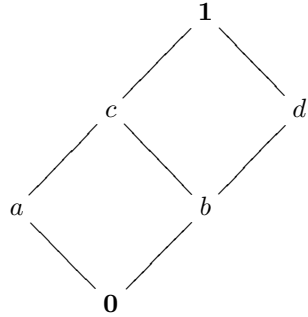
	\[\xymatrix{ & & \mathbf{1}\ar@{-}[dl]\ar@{-}[dr] & \\ & c\ar@{-}[dl]\ar@{-}[dr] & & d\ar@{-}[dl] \\ a & & b & \\ & \mathbf{0}\ar@{-}[ul]\ar@{-}[ur] & & }\]
	Then $C(\mathcal{L})=\{\mathbf{0},\mathbf{1},a,d\}$. Hence $\mathcal{L}$ satisfies $(C_3)$. On the other hand:
	\begin{center}
		$a$ is essential in $[\mathbf{0},a]$,\\
		$b$ is essential in $[\mathbf{0},d]$,\\
		$c$ is essential in $[\mathbf{0},\mathbf{1}]$,\\
		$d$ is essential in $[\mathbf{0},d]$.
	\end{center}
	Thus, $\mathcal{L}$ satisfies $(C_1)$. Therefore, $\mathcal{L}$ is quasi-continuous. Note that $[\mathbf{0},a]$ and $[\mathbf{0},b]$ are isomorphic, but $[\mathbf{0},a]$ is not isomorphic to $[\mathbf{0},d]$.On the other hand, $\mathcal{L}$ is auto-dual, so $\mathcal{L}$ is quasi-discrete. We have that $\mathbf{0}$ is superfluous in $[\mathbf{0},a]$, $b$ is superfluous in $[\mathbf{0},d]$, and $[\mathbf{0},a]\cong[b,d]$. But $[\mathbf{0},a]$ is not isomorphic to $[\mathbf{0},d]$.
\end{example}

%%%%%%%%%%%%%%%%

%\nocite{leemodules,leemodulesdr,leeunit,medina2020abelian}
%%%%%%%

\bibliographystyle{acm}
\bibliography{biblio}

\end{document}